\documentclass[11pt]{article}

\usepackage[left=2cm, right=2cm, bottom=2cm,top=2.5cm]{geometry}
\usepackage{graphicx}
\graphicspath{{../images/},{images/},{../../SIMU/plots/keep/},{../Global_paper/images/}{bioPics/}}

\newif\ifIncludeFastProj %
\newif\iflongversion %

\longversiontrue

\IncludeFastProjtrue

\newif\ifreferonline

\referonlinefalse

\usepackage{amsmath}

\usepackage{subcaption}
\usepackage{bm}

\usepackage{frenchineq}
\usepackage{algorithm,algorithmicx}
\makeatletter
\renewcommand{\ALG@name}{Algo.}
\makeatother
\usepackage{algpseudocode}
\usepackage{cleveref}

\usepackage[colorinlistoftodos,bordercolor=orange,backgroundcolor=orange!20,linecolor=orange,textsize=scriptsize]{todonotes}

\usepackage{multirow}
\usepackage{pbox}
\usepackage{amsfonts}
\usepackage{amssymb}
\usepackage{amsthm}
\usepackage{float}
\usepackage{adjustbox}
\usepackage{tikz}
\usetikzlibrary{arrows,positioning}
\usepackage{enumitem}
\usepackage{mathrsfs} %
\usepackage{dsfont}
\usepackage[numbers,sort&compress]{natbib}

\renewcommand{\citep}[1]{\citeauthor{#1} (\citeyear{#1})}

\newtheorem{theorem}{Theorem}

\newtheorem{proposition}{Proposition}

\newtheorem{assumption}{Assumption}

\newtheorem{example}{Ex.}

\Crefname{algorithm}{Algo.}{Algos.}
\Crefname{proposition}{Prop.}{Props.}
\Crefname{equation}{Eq.}{Eqs.}
\Crefname{figure}{Fig.}{Figs.}
\Crefname{tabular}{Tab.}{Tabs.}
\Crefname{table}{Tab.}{Tabs.}
\Crefname{theorem}{Thm\@.}{Thms.}
\Crefname{definition}{Def.}{Defs.}
\Crefname{section}{Sec.}{Secs.}
\Crefname{assumption}{Assumption}{Assumptions.}
\Crefname{rem}{Rm.}{Rms.}

\newcommand{\txt}{\textstyle}

\newcommand{\hh}{\hspace{-2pt}}

\newcommand{\norm}[1]{\left\Vert #1\right\Vert}

\renewcommand{\hm}{\hspace{-2pt}}

\newcommand{\hmmm}{\hspace{-8pt}}
\newcommand{\SC}{\mathrm{S\hspace{-1pt}C}}
\newcommand{\rr}{\mathbb{R}} %reals set
\newcommand{\argmin}[1]{\underset{#1}{\mathrm{arg}\hspace{-1.3pt}\min}}
\newcommand{\eqDef}{\vspace{-1pt}\overset{\Delta}{=}}%
\newcommand{\eqd}{\eqDef}%

\newcommand{\tr}{\top}
\newcommand{\G}{\mathcal{G}}
\newcommand{\E}{\mathcal{E}}

\newcommand{\dpart}[2]{\dfrac{\partial #1}{\partial #2}}

\newcommand{\dsone}{\mathds{1}}

\newcommand{\X}{\mathcal{X}}
\newcommand{\xx}{\bm{x}} %vector indexed by time periods

\newcommand{\lbxx}{\underline{\xx}}

\newcommand{\ubxx}{\overline{\xx}}
\newcommand{\N}{\mathcal{N}}
\newcommand{\n}{n} %index player
\newcommand{\m}{m}
\newcommand{\T}{\mathcal{T}}
\renewcommand{\t}{t}
\newcommand{\nt}{_{\n,\t}}

\newcommand{\M}{\mathcal{M}}

\renewcommand{\G}{\mathcal{G}}

\renewcommand{\L}{\mathcal{L}} %lines set
\renewcommand{\t}{t} %time index
\newcommand{\nti}{_{\n,\t}}
\newcommand{\mti}{_{\m,\t}}

\newcommand{\mc}[1]{\mathcal{#1}}

\newcommand{\oti}{_{0,t}}

\renewcommand{\P}{\mathcal{P}} %notation for problems
\newcommand{\sol}{\mathrm{sol}}

\newcommand{\uV}{\underline{V}}
\newcommand{\oV}{\overline{V}}
\newcommand{\oP}{\overline{P}}
\newcommand{\uP}{\underline{P}}
\newcommand{\oQ}{\overline{Q}}
\newcommand{\uQ}{\underline{Q}}
\newcommand{\osigma}{\overline{\sigma}}
\newcommand{\usigma}{\underline{\sigma}}

\newcommand{\pRE}{p^\mathrm{p}} %prod RE
\newcommand{\qRE}{q^\mathrm{p}} %prod RE
\newcommand{\pC}{p^\mathrm{c}}
\newcommand{\qC}{q^\mathrm{c}}
\newcommand{\tauC}{\tau^\mathrm{c}}
\newcommand{\rlbRE}{\underline{\rho^\mathrm{p}}} %coeff reactive/active RE
\newcommand{\rubRE}{\overline{\rho^\mathrm{p}}} %coeff reactive/active RE

\newcommand{\upRE}{\overline{P}^\mathrm{p}} %prod RE UB

\newcommand{\llam}{\bm{\lambda}}

\newcommand{\pp}{\bm{p}}
\newcommand{\qq}{\bm{q}}

\newcommand{\A}{\mathcal{A}}
\newcommand{\dlmp}{\lambda^{\text{p}}}
\newcommand{\dlmq}{\lambda^{\text{q}}}
\newcommand{\ddlmp}{\bm{\dlmp}}
\newcommand{\ddlmq}{\bm{\dlmq}}

\newcommand{\na}{n_{\text{-}}} %ancestor node
\newcommand{\nati}{_{\na ,t}} %ancestor node

\newcommand{\pro}{\mathscr{P}}

\begin{document}
\title{DLMP-based Coordination Procedure for Decentralized Demand Response under Distribution Network Constraints}

\author{Paulin Jacquot
\thanks{P. Jacquot is with  the GERAD research center and Polytechnique Montr\'eal, Montr\'eal, Canada. \texttt{paulin.jacquot@polytechnique.edu}. This work was partially  funded by Mitacs Elevation program.}%
}

\markboth{ }%
{Shell \MakeLowercase{\textit{et al.}}: Bare Demo of IEEEtran.cls for IEEE Journals}

\maketitle

\begin{abstract}
Load aggregators are independent private entities whose goal is to optimize energy consumption flexibilities offered by multiple  residential consumers.
  Although aggregators optimize their decisions in a decentralized way, they are indirectly  linked together if their respective consumers belong to the same distribution grid.
  This is an important  issue  for a distribution system operator (DSO), in charge of the reliability of the distribution network,  it has to ensure that decentralized decisions taken do not violate the grid constraints and do not increase the global system costs.
  From the information point of view,the network state and characteristics are confidential to the DSO, which makes a decentralized solution even more relevant.
  To address this issue, we propose a decentralized coordination mechanism between the DSO and multiple aggregators that computes the optimal demand response profiles while solving the optimal power flow problem. The procedure, based on distribution locational marginal prices (DLMP), preserves the decentralized structure of information and decisions, and lead to a feasible and optimal solution for both the aggregators and the DSO.
  The procedure is analyzed from a mechanism design perspective, and different decentralized methods that could be used to implement this procedure are presented.
\end{abstract}

\textbf{Keywords:} 
Decentralized Systems, Demand Response, Mechanism Design, Load Aggregator, Distribution Locational Marginal Prices, AC Optimal Power Flow.

\section*{Introduction}

\paragraph{Context.}%
{T}he management of electricity consumption flexibilities, or \textit{Demand Response} \cite{siano2014demand}, offered by new usages such as electric vehicles and smart appliances, is considered as  a key component of modern electricity systems.
It will help to increase the share of renewable energy production, reduce carbon emissions and ensure the grid stability and resilience.
In this context, aggregators are new actors of the electricity system, whose role is to \emph{aggregate} a large number of individually negligible consumption flexibilities offered by residential or small consumers, and valuate these  flexibilities on the demand response market \cite{gkatzikis2013role} or as a service offered to the system operator.
In competitive electricity markets as in Europe or in the United States, several Load Aggregators (LAs) can be present on the same distribution network, implying the need for coordination.
There is a hierarchical decisions structure, as explained in \cite{gkatzikis2013role}, from the distribution system operator (DSO) in charge of the network, which interacts with the  present LAs, each LA interacting at the lower level  with a subset of affiliated  end-consumers. 
This decentralized system involves multiple actors that are meant to engage in decentralized decisions. 
However, the decisions of LAs  and of the DSO are all linked by the physical constraints of the underlying network (line capacities, voltage limits, etc): if each LA manages its flexibilities %
 ignoring these constraints, the resulting flows could jeopardize the stability of the network or, from the mathematical point of view,  be infeasible.

In addition to the multiplicity of actors, the information asymmetry is also a key issue in the coordination: the network physical parameters (topology, line resistances and capacities, etc.) are often considered as confidential by the DSO and not revealed to third parties, while, on the other hand, LAs might have privacy considerations regarding the flexibilities provided by consumers.

The objective of the present paper is to provide a decentralized coordination mechanism for LAs, based on the computation of  Distribution Locational Marginal Prices (DLMPs) obtained from a conic relaxation of the Alternative Current Optimal Power Flow (ACOPF) problem.
Considering an  AC model of the network enables to take into account not only capacity constraints but also voltage and angle constraints, which are limiting in practice in distribution networks, as stated above.
The underlying idea of the proposed mechanism is to use DLMPs, centrally computed by the DSO, as price incentives for LAs to manage their flexibilities.  %
Through such a procedure, the decentralized structure as well as the asymmetry and privacy of information are preserved.

\paragraph{{Related Works.}}
Recent works have shown the existence of tight conic relaxations of the ACOPF problem, with some cases of exact relaxations  \emph{exact} in particular for radial networks. %
The works \cite{lavaei2012zero,sojoudi2014exactness} consider a Semi-Definite Programming (SDP) relaxation of ACOPF and show its exactness.
The works   \cite{baradar2013second,farivar2013branch,kocuk2016strong} consider different Second Order Cone Programming (SOCP) relaxations of the OPF problem and show  that the SOCP relaxations can be exact under some additional assumptions.
The authors in   \cite{subhonmesh2012equivalence} showed that, in tree networks, the \emph{branch flow } SOCP relaxation is   exact whenever the SDP model \cite{lavaei2012zero} is exact.
Recently, Zohrizadeh \textit{et al} \cite{zohrizadeh2020survey}  proposed a survey on the different relaxation techniques and associated results.
  As SOCP is computationally simpler than SDP \cite{boyd2004convex,gan2014exact}, we focus in this paper on the SOCP branch flow model of \cite{farivar2013branch}.

  In \cite{gatsis2013decomposition}, the authors propose a decomposition algorithm, based on the \emph{Cutting Plane Methods} for the decentralized coordination of several aggregators on a network. Lagrangian multipliers--associated to the aggregated power demand equality constraint of each aggregator--are used to defined  coordination signals, but these prices are not ``locational'', and power flow constraints are not considered.

  Le Cadre \textit{et al} \cite{lecadre2019game} consider the coordination between TSO and DSO, relying on the branch flow formulation and SOCP relaxation for the OPF problem. They compare different solutions obtained from a centralized optimization, a generalized game and a Stackelberg model.
  
  The idea of using DLMPs to coordinate LAs and electric vehicles charging in a decentralized fashion appears in \cite{li2013distribution}. %
  Scott \textit{et al}  \cite{scott2019network} consider the optimization of distributed energy resources (DERs), under the nonconvex power flow constraints of a multiphase unbalanced network. They rely on the alternating direction method of multipliers (ADMM) to solve a distributed receding-horizon OPF and obtain DLMPs, and show results of real experiments of their method.
  In \cite{mhanna2018component}, authors consider the radial formulation of the OPF for distributed generators. Instead of relying on convex relaxations, they study numerically ADMM and dual decomposition algorithms that consider different nonconvex subproblems for branches and buses of the network.
  
  Huang \textit{et al}  \cite{huangDLMP2015}  propose to use a quadratic model to implement the decentralized procedure using DLMPs, to coordinate LAs of electric vehicles or heat pumps. They extend their model,  considering the possibility of negative prices (subsidies) in \cite{huang2019dynamic}.
The papers \cite{li2013distribution,huangDLMP2015,huang2019dynamic} consider a direct current (DC) power flow model, and DLMPs are the Lagrangian multipliers associated to the capacity   constraint for each time period in the DSO optimization problem.  Although DC power flow models are much easier to solve than ACOPF problems,  a DC model  does not consider all existing constraints in distribution networks such as voltage limits.

 Papavasiliou \cite{papavasiliou2017DLMP} considers two formulations of the ACOPF in distribution networks, based on an implicit function formulation and on the SOCP relaxation \cite{peng2016distributed}, and derives expressions and properties of the  DLMPs in each formulation. The SOCP formulation adopted in this paper, is the most interesting from a computational point of view.

 Liu \textit{et al} \cite{liu2014day} consider solving the original (nonconvex) ACOPF problem for the DSO to obtain DLMPs %
  sent to multiple LAs to compute flexible consumption profiles satisfying consumers constraints.
 However, because of its nonconvexity  \cite{farivar2013branch}, the original ACOPF problem  can be very hard to solve in practice, and only a local optimum may be found. 
As the objective functions considered by LAs in \cite{liu2014day} are linear, one cannot ensure, as noted in \cite{huangDLMP2015}, that the combination of the individual LAs solutions obtained through this decentralized procedure does not necessarily correspond to the optimal solution computed by the DSO.

Lin \textit{et al} \cite{lin2019decentralizedAC} propose a coordination mechanism between transmission and distribution level, considering an AC branch flow \cite{farivar2013branch} model for distribution grids. Their method relies on the exchange of information (boundary variables and lower bounds) between transmission and distribution levels.
 
Bai \textit{et al} \cite{bai2017distribution} propose to solve the ACOPF, considering different types of distributed energy resources (DERs) as well as feeder reconfiguration and on load tap changers, involving discrete decision variables. 
The authors propose to solve in a first step the  problem considering the SOCP relaxation defined in \cite{farivar2013branch} with the discrete variables, to determine the optimal values of these variables. 
In a second step, the discrete variables are fixed and a linearization of the SOCP problem around the optimal value is solved  to obtain DLMPs, which are then transmitted to the DERs.\\
 
The framework considered in this paper differs from the above mentioned works in at least two points.
First, the procedures proposed in the existing literature are not decentralized as, to compute the DLMPs by solving the OPF, the authors consider that the DSO has  access to all the information from the DERs (capacities, state of charge for batteries, power bounds for LAs). 
A solution could be to impose to all the DERs to provide the necessary information to the DSO in time, or that the DSO computes an approximation of these parameters. 
However, another solution analyzed in this paper, is to rely on a \emph{decentralized} iterative procedure in which each DER updates its individual decisions according to a partial information received from the DSO, and then transmit its profile back to the DSO. 
This solution has the advantage to converge to a feasible power flow solution, which may not necessarily be the case if the DSO simply transmits the DLMPS to DERs, as multiple individual profiles can emerge from the linear programs solved by the DERs \cite{huangDLMP2015}.

Second, most of the works  \cite{li2013distribution,liu2014day,huangDLMP2015,bai2017distribution} consider either DCOPF or linearizations of ACOPF. In this paper we obtain the DLMPs directly from the resolution of the branch flow SOCP relaxation of the OPF problem \cite{farivar2013branch}.

 \paragraph{{Main Contributions. }}
The main  contributions brought by this paper are the following:
\begin{itemize}[wide]
\item we  present a model of the interactions of a DSO and multiple LAs subject to the distribution network constraints, highlighting the need for a coordination mechanism between those actors (\Cref{sec:modelAggsNetwork});
\item  we review the existing results on the exactness of the SOCP relaxation for the OPF problem based on the branch flow model, and we consider their extensions to multi-time periods framework (\Cref{thm:farivarSOCP}, \Cref{thm:ganSOCPexactV});
\item  we bring more insight and give additional results on the DLMPs obtained from the optimal solution of the SOCP relaxation  based on the branch flow model (\Cref{prop:DLMPnoRes}, \Cref{prop:dlmpsSubgrad}). In particular, \Cref{prop:DLMPnoRes} shows the link between the DLMPs at one node and at the ancestor node.
Despite of the basic nature of these results, we are not aware of other works where they are stated for this framework;
\item  we provide a DLMPs-based decentralized coordination procedure (\Cref{proc:coordination}) between the DSO and the LAs that respects the decentralized decision and information structure: flexibilities are computed and managed by LAs from DLMPs incentives imposed by the DSO, and the final  obtained profiles result in an optimal solution of the ACOPF problem that can be computed by the DSO. To our limited knowledge, this is the first paper to consider decentralized methods to solve the ACOPF branch flow model considering demand response;
\item we consider the application of standard decomposition methods as Dual Ascent (\Cref{algo:dualAscent}) and ADMM (\Cref{algo:ADMM}) and show how they can be used within the proposed  coordination procedure to respect the above mentioned decentralized features.
  We  propose a novel decomposition method (PDGS, \Cref{algo:dualDecompAltMin}) that has the practical advantage to ensure primal feasibility at each iteration, and we  prove its convergence (\Cref{thm:GScvg});
\item we analyze the proposed procedure based on DLMPs from a mechanism design perspective. We show, through a counter-example, that the DLMPs-based mechanism is not incentive compatible in general (\Cref{ex:counterexampleIncentive}). To our limited knowledge, this is the first time where this result is stated for this framework.
  We also compare the DLMPs-based mechanism to the standard VCG mechanism, a mechanism which satisfies the incentive-compatible property, and explain how VCG could be implemented in a decentralized way;
  
\item last, we provide a numerical illustration of the proposed framework, based on the extension of the 15 nodes network example considered in \cite{papavasiliou2017DLMP} to a multi-time periods and flexible consumption framework.
  We compare numerically the application of decentralized algorithms such as ADMM and PDGS, and compare the payments obtained under the DLMPs-based mechanism and VCG mechanism.
  \end{itemize}

  \paragraph{Structure.} The remaining of the paper is organized as follows. In \Cref{sec:modelAggsNetwork}, we detail the model of the network constraints, load flexibility aggregator and objectives.
  In  \Cref{sec:exactnessSOCP}, we state two theoretical results ensuring the exactness of the SOCP relaxation for the OPF problem in multi-time periods.
  In  \Cref{sec:DLMPs}, we analyze the DLMPs obtained as optimal dual variables of the SOCP relaxation, and state different properties on these values.
  In \Cref{sec:coordMethods}, we present different distributed algorithms that can be used to implement the proposed DLMPs-based decentralized coordination procedure.
  In \Cref{sec:mechanisms}, we present the proposed procedure from a mechanism design perspective, and we give a counter-example showing that this procedure is not incentive compatible. We present how the VCG mechanism could be used in a decentralized way.
Last,  \Cref{sec:numericalStud} is devoted to numerical illustrations of the proposed procedure and using two different numerical methods (ADMM and PDGS). We consider a 15-nodes network that was already considered in the literature, and consider two time periods.
 \medskip

\textsc{Notation.}  We use \textbf{bold} font to denote  vectors (e.g. $\xx$) as opposed to  scalars (e.g. $x$).
Sets are denoted by calligraphic letters (e.g. $\T, \N, \A$), except $\pro$ which is used to label optimization problems.
For a problem $\pro$, $\mathrm{sol}(\pro)$ denotes the set of optimal solutions of $\pro$.

\section{Aggregators on Distribution Network}

\label{sec:modelAggsNetwork}

We consider a distribution network represented by a set of nodes $\N \eqd \{1,\dots,N\}$ and given as a graph $\G \eqd (\N_+,\E)$, where $\N_+ \eqd \N \cup \{0\}$ with $0$ denoting the root node and corresponding to the feeder node (link with transportation network). 

Distribution networks are usually designed such that there is no cycle in the electricity lines. Thus, the graph $\G$ is a tree.

Each node $n\in\N$ corresponds to an individual household, a group of households or a commercial building linked to the distribution network.  
We assume that the operation of the distribution network and the management of flexibilities  is done on a common time horizon $\T$  (e.g. a day), given as a finite set of discrete time periods: \begin{equation*}
\T \eqd \{1, \dots, T\} .
\end{equation*}

\subsection{Branch Flow Model and SOCP relaxation}

We use the branch flow model and the SOCP relaxation introduced in \cite{peng2016distributed} and also considered in \cite{papavasiliou2017DLMP}.

Following \cite{papavasiliou2017DLMP}, we use $p_n$, $q_n$ to denote active and reactive power \emph{consumption} at node $n$: thus $p_n <0$ means that there is production at bus $n$. At the root (feeder) node $n=0$, we expect that some power will be produced (or  bought from the market) and that there is no consumption:  we therefore assume  $p\oti\leq 0$.
Variable $v_n$ stands for  the squared voltage magnitude at bus $n$.
 Variables  $f_n$, $g_n$ and $\ell_n$  denote  the active and reactive power flows and  the squared current magnitude on the line from $n$ to the unique ancestor of node $n$, denoted by $\na$.
\begin{figure}[!ht]
\centering
\begin{tikzpicture}
\node[ draw,circle] (n0) at (-4,0) {$0$};
\node[ draw,circle] (na) at (-2,0) {$\na$};
\node[draw] (r) at (0,0) {$R_n,X_n$};
\node[ draw,circle] (n) at (2,0) {$n$};
\draw [thick] (na) -- (r) -- node [above,xshift=0cm] {$\underset{\leftarrow}{f_n,g_n}$ }   node[below] {} (n) ; 

\draw [dashed] (n0) --(-3.2,0) ;
\draw [dashed] (na) --(-2.8,0) ;
\draw [dashed] (n) --(3.7,1) ;
\draw [dashed] (n) --(3.7,-1) ;
\draw [-latex'] (n) -- node [above,rotate=-90] {$ \ \ p_n,q_n$} (2,-1);

\draw [thick] (-2,-2)-- (2,-2) ;
\node[draw,rotate=-90] (gb) at (-1.2,-1) {$  G_n,B_n$} ;
\draw [thick] (-1.2,0) -- (gb) --  (-1.2,-2);
\end{tikzpicture}
\caption{$\Pi$-model and notations for line from node $n$  to $\na$}
\label{fig:pimodelNotations}
\end{figure}
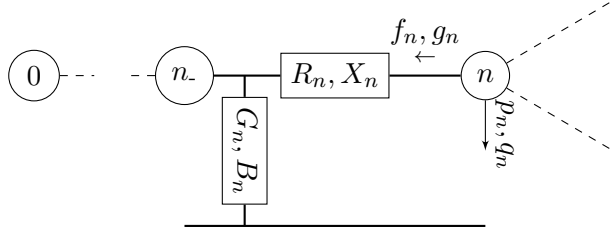
The resistance and reactance on this line are denoted $R_n$ and $X_n$, while the shunt conductance and susceptance at node $n$ are denoted $G_n$ and $B_n$. The power flow magnitude limit (line capacity) on line $(n,\na)$ is denoted by $S_n$. \Cref{fig:pimodelNotations} recalls the notation used to describe the network parameters.

We obtain the set of \emph{branch flow} equations:
\begin{small}
\begin{subequations}
\label{cons:OPF}
\begin{align}
& v\nti \hh - \hh 2(R_nf\nti \hh + \hh X_n g\nti ) + \ell\nti (R_n^2 \hh + \hh   X_n^2) = v_{\na,t}, \   \forall n \in \N  \label{eq:DSO-voltageLines}\\ 
 & f\nti \hm -\hmmm \sum_{m: m_{-}=n}\hmmm \big(f\mti \hh- \hh \ell\mti R\mti \big)\hh +\hh p\nti  \hh +G_n v\nti= 0 , \forall n \in \N_+   \label{eq:DSO-actPowBal}\\
&  g\nti \hh - \hmmm \sum_{m:m_{-}=n} \hmmm \big(g\mti  \hh-  \hh \ell\mti X_m\big)  + q\nti \hh - \hh B_nv\nti=0 ,\  \forall n \in \N_+  \label{eq:DSO-reactPowBal}\\
  & f\nti^2 + g\nti^2 \leq v\nti \ell\nti \ , \ \forall n \in \N  \label{eq:DSO-currentmagn}\\
  & f\nti^2 + g\nti^2 \leq S_n^2 \ , \ \forall n \in \N   \label{eq:DSO-flowCons} \\
  & (f\nti-R_n\ell\nti)^2 + (g\nti-X_n\ell\nti)^2 \leq S_n^2 \ , \ \forall n \in \N   \label{eq:DSO-flowCons2} \\
 & \uV_n \leq v\nti \leq \oV_n \ , \ \forall n \in \N_+     \label{eq:DSO-voltageLim} 
\end{align}  
\end{subequations}
\end{small}
and we further denote by $ \beta\nti,  \dlmp\nti  , \dlmq\nti   \in \rr $ the Lagrangian multipliers associated respectively to \eqref{eq:DSO-voltageLines}, \eqref{eq:DSO-actPowBal}, \eqref{eq:DSO-reactPowBal}, and by $\gamma\nti ,\eta\nti^+,\eta\nti^-, \usigma\nti,\osigma\nti \geq 0 $ the ones associated to constraints \eqref{eq:DSO-currentmagn}, \eqref{eq:DSO-flowCons}, \eqref{eq:DSO-flowCons2}, \eqref{eq:DSO-voltageLim}.
In \eqref{cons:OPF}, the equality $ f\nti^2 + g\nti^2 = v\nti \ell\nti  $ defining the current magnitude $\ell\nt$ in the actual branch flow model, is \emph{relaxed} as the inequality \eqref{eq:DSO-currentmagn}, which is a cone constraint (see \cite{farivar2013branch} for more details), defining what we refer to as the SOCP relaxation for the ACOPF problem.

\subsection{Electricity Load  Aggregators}

We consider that a set $\A$ of several \emph{Load Aggregators} (LAs) coexist on the distribution network.  
Each LA $a \in \A$ manages a subset $\N_a$ of the nodes in the network, such that $\bigcup_a \N_a$ forms a partition  of $\N$ (each node is affiliated with one LA).

The active and reactive power at each node is not fixed but \emph{flexible}: the LA $a\in \A$  manages the \emph{flexible} net power consumption  $(p\nt, q\nt)_{t\in\T} $ of individual consumers at $n$ for each node $\n \in \N_a$, w.r.t. individual constraints. 
We consider the possible presence  of local renewable energy sources (e.g. photovoltaic panels) that can also be managed by the LA, such that $p\nt$ is composed of a production part $\pRE\nt \geq 0$ and a consumption part $\pC\nt \geq 0$, as:
\begin{subequations}
\label{cons:localCons}
\begin{equation} \label{eq:prod-cons-defp}
\forall n \in \N, \forall t \in \T,\  p\nt= \pC\nt-\pRE\nt \ .
\end{equation}
 The  constraints on the consumption of each node $n$ are described through a global energy demand $E_n$ over the $T$ time periods, as well as lower and upper bounds for each time period, that is:
\begin{align} 
& \txt\sum_{\t\in\T} \pC\nt \geq E_n \label{cons:totalDemand}\\
 & \uP\nti \leq \pC\nti \leq \oP\nti \ , \label{cons:boundsConsum}
\end{align}
and we denote by $\alpha_n\geq 0 $ (resp. $\underline{\nu}\nti, \ \overline{\nu}\nti$) the Lagrangian multipliers associated to \eqref{cons:totalDemand} (resp. \eqref{cons:boundsConsum}.
Constraints \eqref{cons:totalDemand}-\eqref{cons:boundsConsum} give a simple model for deferrable loads such as electric vehicles and water heaters, which has been widely used  \cite{PaulinTSG17,mohsenian2010autonomous,li2011optimal,chen2014autonomous,baharlouei2014efficiency,samadi2012advanced}.

 In practice, each LA can aggregate the flexibilities offered by individual end-consumers to fit in the model  \eqref{cons:totalDemand}-\eqref{cons:boundsConsum}. 
 The same idea of aggregate set of constraints in a particular form is formulated for instance in  \cite{muller2017aggregation}, where the authors propose an aggregation procedure considering \emph{zonotopic sets} instead of the simplex structure   \eqref{cons:totalDemand}-\eqref{cons:boundsConsum}.
 
We consider that active and reactive consumptions  are correlated by a given ratio (depending on the type of appliances) as:
 \begin{align} \label{cons:relationActiveReactiveCons}
  & \qC\nt= \tauC_n \  \pC\nt  \ .
 \end{align}
Each LA also controls the power $(\pRE\nt)_{n\in\N_a,t}$ produced by distributed energy resources (DERs) across its affiliated nodes $\N_a$. 
The DERs \cite{bai2017distribution} we consider are local and renewable energy (either photovoltaic or wind power) installed in households. 
We consider that the power $\pRE\nt$ produced at period $t$ is upper bounded by an intermittent available power $P\nt$ (which depends on the DERs capacities and the wind or sun power), and can be adjusted within $[0,P\nt]$. 
The associated produced reactive power $\qRE\nt$ can be adjusted through the control of smart inverters  \cite{farivar2012optimal}, which gives:
\begin{align}
& 0 \leq \pRE\nt \leq  \upRE  \label{cons:boundsProd} \\
& \rlbRE\nt \pRE\nt \leq \qRE\nt \leq  \rubRE\nt \pRE\nt \label{cons:boundsReactiveProd}  \ .
\end{align}
\end{subequations}
We consider that DERs are used at cost zero by the LA operating them: incentives and costs can be distributed by the LA to its affiliated households in a separate process. The study of interactions between an LA and its affiliated consumers is beyond the scope of this paper.

We define the  cost function of  LA $a$ as $\phi_a(\cdot)$ and assume that it depends on the active consumption profiles $ (p_n)_{n\in\N_a} $.
 The cost function $\phi_a$ is determined by exogenous parameters (e.g. from   day-ahead electricity market prices) %
or by some incentives from the system operator.
Thus, the \emph{local} problem $\mathcal{P}_a$ of LA $a$ is formulated as:
\begin{align} \label{pb:operator_local}     \tag{$\pro_a$}
&     \min_{\xx_a=(\pp_n,\qq_n)_{n\in\N_a} } \phi_a\big( (\pp_n)_{n\in\N_a} \big) \\
&    \text{s.t. } \  \eqref{cons:localCons}  \nonumber \,
  \end{align}
  where $\xx_a \eqd (\pp_n,\qq_n)_{n\in\N_a}$ denote the variables of LA $a\in\A$ that are controlled locally by $a$, and $\pp_n\eqd (p\nt)_{t\in\T},\ \qq_n \eqd(q\nt)_{t\in\T}$.

\subsection{Distribution System Operator}

The Distribution System Operator (DSO) is a central, independent entity in charge of the operation of the distribution network, which is able to impose (price) incentives to LAs.

In most of the electric systems, the DSO also bears the costs of the network losses, by buying the necessary energy at the day-ahead market price. The corresponding fees are usually recovered by taxes to utilities collected by the DSO. %

We assume the existence of underlying energy price functions  $\big(c_t(.)\big)_{t\in\T}$  modeling, for instance, an electricity market or some production costs. 
We consider different possible objective functions  for the DSO:
\begin{enumerate}[label=\roman*),wide]
\item minimization of  the \emph{social cost} of LAs (in this case, the DSO has no proper cost function):
\begin{subequations} \label{eq:possibleObj}
\begin{equation}\label{eq:defSC}
 \SC(\pp) \eqDef  \sum_{a\in\A} \phi_a\big( (p\nt)_{nt}\big)= \sum_{t\in\T} c_t\big( \txt\sum_{n\in\N}   p\nt \big) \ ;
\end{equation}
\item minimization of the total active power injected at the substation node 0 into the  grid (e.g.  bought at the DA market price):
\begin{equation}\label{eq:defcostExtpower}
\phi_{\text{inj}}(\pp_0) \eqDef \sum_{t\in\T}c_t( \pRE_{0t})  =  \sum_{t\in\T}  c_t(-p_{0t} )  \ ,
\end{equation}
where, following  \cite{huangDLMP2015} and others, we assume that for each time period $t$, there is an energy cost function which is  affine and  increasing, i.e.,  $c_t(x) \eqd \alpha_t x + \beta_t x^2$ for some $ \alpha_t, \beta_t \geq 0$.
\item minimization of active network losses:
\begin{equation}\label{eq:defcostExtpowerLoss}
\phi_\mathrm{loss}(\bm{\ell}) \eqDef \sum_{t\in\T}  \sum_{n\in\N}  R_n \ell\nti \ .
\end{equation} 
\end{subequations}
\end{enumerate}
From a multi-agent point of view, the objectives concern different entities:  the cost $\phi_a$ in \eqref{eq:defSC} concerns the LA $a$, while the costs \eqref{eq:defcostExtpower} and \eqref{eq:defcostExtpower} make more sense at the level of the DSO.
It is thus relevant to differentiate both cases by the following notation:
\begin{itemize}
\item $\xx_\A \eqd(\xx_a)_{a\in\A} $ denotes the LAs variables as defined above, while $\phi_\A(\xx_\A) \eqd \sum_a \phi_a(\xx_a)$ denotes the LAs part in the cost function;
  \item $\xx_0  \eqd (\bm{p}_0,\bm{q_0},\bm{v},\bm{\ell},\bm{f},\bm{g}) $ denotes the DSO variables, while $\phi_0(\xx_0) \eqd \phi_{\text{inj}}(\pp_0) +  \alpha_{\mathrm{loss}}\phi_\mathrm{loss}(\bm{\ell}) $ denotes the DSO part in the cost function.
\end{itemize}
In what follows, we will assume that the DSO objective function $\Phi$  will be either $\phi_\A$ (as the minimization of social cost) or $\phi_0$, or  a linear combination of the three objectives \eqref{eq:possibleObj}:

\begin{equation*} 
\Phi(\xx) \eqd  \phi_0(\xx_0) +  \phi_\A(\xx_\A) \ .
  \end{equation*}
  
  In the decentralized framework adopted of this paper, the DSO  considers the local variables $\xx_\A \eqd (p_n,q_n)_{n\in\N}$ as fixed parameters, as those variables are managed by the LAs. Moreover, the DSO has no access to the cost function $\phi_a$ of each LA $a$ and does not consider these costs in its optimization.
  Thus, the DSO faces the following optimization problem:
\begin{equation} \label{pb:DSO-OPF} \tag{$\pro_0(\xx_A)$}
\begin{split} 
  & \min_{\xx_0= (\bm{p}_0,\bm{q_0},\bm{v},\bm{\ell},\bm{f},\bm{g})} \phi_0(\xx_0) \\ & \text{s.t. } \  \eqref{cons:OPF} \ .
\end{split}
\end{equation}
Despite he does not control  the local  variables $\xx_\A$, the DSO is interested in finding the solution that is socially optimal for the whole system, that is, obtaining an optimal solution of the global \emph{centralized} problem:
\begin{align} \label{pb:OPFc} \tag{$\pro^{\star}$}
 &   \min_{\xx_0,\xx_\A} \Phi(\xx) \\
  &\text{s.t.} \ (\ref{cons:OPF} - \ref{cons:localCons}) \ . \nonumber
\end{align}
From now on, we assume that  problem  \eqref{pb:OPFc} has a solution:
\begin{assumption}
  \label{assp:exiSsolution}
  There exists $(\xx_0, \xx_A)$ satisfying (\ref{cons:OPF} - \ref{cons:localCons}) or, equivalently, problem \eqref{pb:OPFc} is feasible.
\end{assumption}
In what follows, we recall that the SOCP relaxation \eqref{pb:OPFc} can be exact: a solution of this problem will, in many cases, give an optimal solution of  the original OPF problem (i.e. the same problem \eqref{pb:OPFc} with constraint \eqref{eq:DSO-currentmagn} written as an equality).

\section{Exactness of SOCP relaxation}
\label{sec:exactnessSOCP}
In \cite{farivar2013branch,gan2014exact}, the authors show that, \emph{in a radial network}, the SOCP relaxation \ref{pb:OPFc} of the nonlinear OPF problem 
is exact under some specific assumptions.
The first result stated below is a straightforward extension of \cite[Thm.~1]{farivar2013branch} to the multi-time periods OPF problem \eqref{pb:OPFc}.

\begin{theorem} \label{thm:farivarSOCP}
  Suppose that the objective function $\Phi$ is convex, strictly increasing in $\ell$, independent of $f,g$, plus one of the following:
  \begin{itemize}
  \item $\Phi$ is \emph{nonincreasing} in  $\pC,\qC>0$, and \emph{upper} bounds \eqref{cons:boundsConsum},\eqref{cons:relationActiveReactiveCons} on $\pC, \qC$ are not binding (that is, $\oP\nti=\infty, \oQ\nti=\infty$);
 \item    $\Phi$ is \emph{nondecreasing} in $\pRE,\qRE >0$, and \emph{lower} bounds \eqref{cons:boundsProd}, \eqref{cons:boundsReactiveProd} on $\pRE,\qRE$ are not binding;
    \end{itemize}
then the  SOCP relaxation \eqref{pb:OPFc} given above is exact.
\end{theorem}

The authors in  \cite{gan2014exact} also consider hypotheses on the voltage magnitude constraints that should not be binding on a strong sense.
\Cref{thm:ganSOCPexactV} below is an immediate extension of \cite[Thm.1]{gan2014exact} to multi-time periods.
\begin{theorem} \label{thm:ganSOCPexactV}
    Let $\P_n=\{n , \na, \dots ,0 \} \subset \N_+$ denotes the unique path from $n$ to root node $0$.
  If $\Phi(.)$ is strictly increasing in $p_0$ and if:
  \begin{itemize}[wide]
    \item there is no shunt capacitances and admittances ($\forall n \in \N, \ B_n=G_n=0$);
    \item    for any optimal  $\xx=(\bm{p}, \bm{q}, \bm{f}, \bm{g}, \bm{v}, \bm{\ell}) \in \sol ($\ref{pb:OPFc}$)$ ,   the \emph{linearized} flow solutions (considering $\ell\nti=0$ in \eqref{cons:OPF}):
    \begin{align*} \textstyle
      \hat{f}_n(\pp_t) \eqd  - \sum_{m:n\in \P_m} p\mti,\quad \hat{g}_n(\bm{q}_t) \eqd  - \sum_{m:n\in \P_m} q\mti \ , \\
\ \hat{v}\nti(\xx) \eqd V_0 + 2 \txt\sum_{m \in \P_n} R_m\hat{f}_m(\pp_t) +  X_m \hat{g}_m(\bm{q}_t )
  \end{align*}
verifies: \hspace{1cm}
$ \hat{v}\nti(\xx)     < \oV_n  $,
\hfill 

    \item for any $t\in\T$, any path $(n_1,\dots,n_l)$ to a leaf bus $l\in\N$, and any $(1 \leq s \leq k \leq l)$, we have:
$$ \underline{A}_{n_s,t} \dots   \underline{A}_{n_{k-1},t} u_{n_k} > 0  \text{ with } u_{n} \hh \eqd  \hh \begin{pmatrix} R_n \\ X_n
\end{pmatrix} $$
$$ \text{ and } \underline{A}_{n,t} \hh \eqd  \hh I-\frac{2}{\uV_n} \hh  \begin{pmatrix} R_n \\ X_n 
\end{pmatrix} \hh \begin{pmatrix}
 [ \hat{f}_n(\bm{\uP}_t)]^+ \hh & [ \hat{g}_n(\bm{\uQ}_t)]^+ %
\end{pmatrix}  $$
\end{itemize} 
then the  SOCP relaxation \eqref{pb:OPFc} given above is exact.
\end{theorem}

It has been verified in \cite{gan2014exact} that the conditions of \Cref{thm:ganSOCPexactV} are verified in many standard networks.
More importantly, the conditions of \Cref{thm:ganSOCPexactV,thm:farivarSOCP} are \emph{sufficient} conditions, but are not necessary: as shown in the example of \Cref{sec:numericalStud}, but also in \cite{gan2014exact}, in general the SOCP relaxation \eqref{pb:OPFc} is exact even if conditions of \Cref{thm:ganSOCPexactV,thm:farivarSOCP}  do not hold.

The objective of this paper is not to improve those results of exactness, but to rely on  the SOCP relaxation of the branch flow model  to design an efficient coordination procedure between the DSO and LAs while considering the decentralized information structure.
Thus, for the remaining of the paper, we make the following assumption.
\begin{assumption} \label{assp:exactness}
For the instances considered,  problem \eqref{pb:OPFc} is an exact relaxation of the actual ACOPF problem.
\end{assumption}
\Cref{assp:exactness} is easy to verify a posteriori for a solution: it suffices to check that \eqref{eq:DSO-currentmagn} is an equality.

\section{Distribution Locational Marginal Prices}
\label{sec:DLMPs}
The idea of using %
DLMPs has been considered in the literature in the last five years \cite{huangDLMP2015},\cite{bai2017distribution},\cite{papavasiliou2017DLMP}. 
In addition to being mathematically funded, DLMPs provide a  decentralized tool  that has shown its efficiency for more than a decade at the transmission level.

In a context of an important level of local renewable energy, DLMPs for the \emph{reactive} power part  \cite{lin2019decentralizedAC} can also be a valuable tool to improve the stability and power quality of the distribution grid.
In some cases, the reactive power can be adapted locally, for instance through smart  inverters associated to renewable sources.
Thus, in this paper, we will consider DLMPs for both the active power  $p_n$ and the reactive power $q_n$ for node $n$, as defined in \Cref{sec:modelAggsNetwork}.

Besides,  advances in conic optimization to solve the OPF problem (SOCP \eqref{pb:OPFc} and SDP \cite{Sojoudi2016} relaxations) make the idea of relying on DLMPs for decentralized coordination more relevant.
Indeed, by solving the SOCP problem \eqref{pb:OPFc} or \eqref{pb:DSO-OPF} representing an instance of OPF with standard interior point methods \cite[Ch.~11]{boyd2004convex},  the active (resp. reactive) DLMPs are directly obtained by the system operator as the dual solutions $(\ddlmp_n)_{n\in\N}$ (resp. $(\ddlmq_n)_{n\in\N}$), and can be transmitted to the LAs as price incentives.
This is the basis of the decentralized coordination methods proposed in \Cref{sec:coordMethods}.

The DLMPs $(\ddlmp_n,\ddlmq_n)$ emerging as Lagrangian multipliers of the SOCP formulation \eqref{cons:OPF} can be decomposed and interpreted.  \Cref{prop:expressDLMPKKT} completes \cite[Prop.~3.2]{papavasiliou2017DLMP} where the author studies different interpretations of DLMPs in the SOCP formulation \eqref{cons:OPF} of the ACOPF, as well as in two other ACOPF formulations.
\begin{proposition} \label{prop:expressDLMPKKT}
The DLMPs $(\ddlmp_n,\ddlmq_n)$ at node $n$ can be expressed as a linear combination of the DLMPs of the ancestor node $\na$, in addition to some dual quantities related to the line capacity constraints \eqref{eq:DSO-flowCons},\eqref{eq:DSO-flowCons2} and the voltage definition constraints \eqref{eq:DSO-voltageLines},\eqref{eq:DSO-currentmagn}:
  \begin{small}
  \begin{align*}
\ddlmp_n &= \ddlmp_{\na}  -2 \bm{f}_n\bm{\gamma}_n -  2 \bm{f}_n\bm{\eta}_n^+  -  2 (\bm{f}_n  - R_n\bm{\ell}_n )\bm{\eta}_n^-  +2\bm{\beta}_n  R_n  , \\
 \ddlmq_n & = \ddlmq_{\na} - 2 \bm{g}_n\bm{ \gamma}_n -  2  \bm{g}_n\bm{\eta}_n^{+} -  2(\bm{g}_n- X_n\bm{\ell}_n)\bm{\eta}_n^-   +2\bm{\beta}_n  X_n  .
  \end{align*}
  \end{small}
\end{proposition}
\begin{proof}
The Lagrangian function associated to problem \eqref{pb:DSO-OPF} is:
\begin{equation}
\begin{split}
\mathcal{L}= &  \phi_0(\xx) +  \sum_{\t\in\T} \sum_{\n\in\N} \Big[ \eta\nti^+ \times \Big(f\nti^2 + g\nti^2 - S_n^2 \Big)   \ 
+   {\eta\nti^-} \times \Big( \big(f\nti-R_nl\nti\big)^2 +\big(g\nti-X_nl\nti\big)^2 - S_n^2 \Big)    \\
+ & \beta\nti\times \Big( v\nti-2(R_nf\nti+X_n g\nti )+ l\nti (R_n^2+X_n^2) - v\nati  \Big)  \ 
+   \gamma\nti \times \Big( f\nti^2 + g\nti^2 - v\nti l\nti  \Big) \\
+ & \lambda\nti\times \Big( f\nti - \sum_{m \in \delta_n^+} \big(f\mti-l\mti R\mti \big) + p\nti  +G_n v\nti \Big)  \\
+ &   \mu\nti \times \Big(  g\nti - \sum_{m\in\delta_n^+} \big(g\mti - l\mti X_m\big) +q\nti-B_nv\nti \Big)  \\
+ & \usigma\nti \times ( \uV_n - v\nti ) + \osigma\nti \times (v\nti - \oV_n) \  .
\end{split}
\end{equation}
We can then obtain the equalities of \Cref{prop:expressDLMPKKT} from the  KKT conditions of optimality, as $\dpart{\L}{f_n}=0$ and $\dpart{\L}{g_n}=0$.
  \end{proof}
  \Cref{prop:expressDLMPKKT}  does not provide a closed form of the DLMPs, as we cannot obtain an explicit expression of   Lagrangian multipliers $\bm{\gamma}_n, \bm{\beta}_n$.  The interpretation of multipliers $\bm{\gamma}_n, \bm{\beta}_n$ is not straightforward, as noticed in \cite{papavasiliou2017DLMP}, where the author states that the DLMPs  $\ddlmp_n$ can be expressed as nonlinear  functions of $ \ddlmp_{\na} ,  \ddlmq_n, \ddlmq_{\na} ,\bm{\eta}_n^{+}$ and $\bm{\eta}_n^{-}$, although these functions are implicit. 
  However, \Cref{prop:expressDLMPKKT} shows that the DLMPs at a given node are linearly linked to the DLMPs at the parent's node.
This is further highlighted in \Cref{prop:DLMPnoRes} below where it is shown that, if reactances are negligible, then the DLMPs at one node are equal to the DLMPs at the ancestor node.
  \begin{proposition}\label{prop:DLMPnoRes}
In the limit of negligible reactances and shunt reactances at node $n$, i.e $$R_n, \ X_n, \ G_n, \ B_n \longrightarrow 0 $$  and if the line $(n,\na)$ is not saturated at time $t$ (i.e. inequalities \eqref{eq:DSO-flowCons},\eqref{eq:DSO-flowCons2} are strict), then the DLMPs at $n$ at $t$  and at the parent node $\na$ are equal: 
    \begin{equation} \label{eq:DLMPatancestor}
      \ | \dlmp\nti - \dlmp\nati| \longrightarrow 0 \ ,\ \  |\dlmq\nti -  \dlmq\nati|\longrightarrow 0
     \ .
      \end{equation}
    \end{proposition}
    \ifreferonline
The proof is given in the online version\cite{paulinjacquot2020dlmp}.
    \else
    \begin{proof}  As capacity constraints are not binding, we get from the complementarity conditions that $\eta\nti^+= \eta\nti^-=0$. From the KKT condition obtained by taking the derivative of the Lagrangian w.r.t $\ell\nti$ and as   $R_n=X_n=G_n=B_n=0$, we get:
      \begin{equation*}
0= \gamma\nti v\nti \ .
\end{equation*}
As $v\nti > 0$, we necessarily have $\gamma\nti=0$. Thus,  simplifying the equalities stated  in  \Cref{prop:expressDLMPKKT} gives exactly \eqref{eq:DLMPatancestor}.
\end{proof}
\fi
      \Cref{prop:DLMPnoRes} has further consequences: if the conditions hold for all nodes $n\in\N$, then all DLMPs are equal to the DLMPs at the root node $\ddlmp_0,\ddlmq_0$. In particular, for each $n\in\N, \ddlmp_n= \nabla_{ \bm{\pRE}_0}\phi_{\text{inj}}(\bm{\pRE}_0) \ $  i.e. the DLMPs are all equal to the root node marginal production cost, which is what we expect to obtain.
      If resistances are nonzero and capacity constraints are saturated, the DLMPs will deviate from this value to account for the costs of losses and congestion effects.

      \bigskip

Let us reformulate problem \eqref{pb:OPFc} formally to consider the decomposition between the DSO and LAs, as:
\begin{subequations}
  \label{pb:formalComplete}
\begin{align}
& \min_{\xx}\Phi(\xx) \\
 & A_0 \xx_0 + B \xx_\A = \bm{b}& (\llam)   \label{cs:formalCouplingLin} \\
  & %
    \xx_0 \in \X_0 \label{cs:formalx0}\\
  & %
    \xx_\A \in \X_\A \label{cs:formalxA},
\end{align}
\end{subequations}
where $\xx \eqd (\xx_0, \xx_\A) $ and where
\begin{itemize}[wide]
  \item $\xx_\A \eqd (\xx_a)_{a\in\A}$ denotes the LAs (consumption) variables with feasibility set $\X_\A$ with 
\begin{equation*}
\xx_a \eqd (p\nt,q\nt)_{n\in\N_a,t\in\T}
\end{equation*}
related to her affiliated nodes (and the corresponding consumption/production variables);
\item $\xx_0$ denotes the DSO (operation) variables with feasibility set $\X_0$, that is:
\begin{equation*}
\xx_0 \eqd (\bm{p}_0,\bm{q_0},\bm{v},\bm{\ell},\bm{f},\bm{g}) \ .
\end{equation*}
\end{itemize}

Constraint \eqref{cs:formalCouplingLin} refers to the coupling constraints between LAs and DSO variables \eqref{eq:DSO-actPowBal}.
It is assumed that matrix  $B$ is block diagonal $B=\text{diag}(B_a)_{a\in\A}$, where $B_a\in \mathcal{M}_{k_a,n_a}(\rr)$.
We denote by  $k\eqd \sum_a k_a$ the dimension of the LAs variables, thus  $A_0 \in \M_{k,n_0}(\rr)$.
From now on,  the notation  $\llam  \in \rr^{k} $ is used  to denote the complete vector of Lagrangian multipliers associated to \eqref{cs:formalCouplingLin}, corresponding to $(\ddlmp,\ddlmq)$ in Problem  \eqref{pb:OPFc}.

From \eqref{pb:formalComplete}, we derive:
\begin{align}
  \label{eq:decompLag}
  \begin{array}{l}
 \min_{\xx_0\in\X_0,\xx_\A\in\X_\A}\Phi(\xx) \\ 
 \text{s.t. } A_0 \xx_0 + B \xx_\A = \bm{b}    
  \end{array} =
  \min_{\xx_\A \in\X_A}\phi_\A(\xx_\A) +   \begin{array}{l} \min_{\xx_0 \in \X_0} \phi_0(\xx_0)\\
                            \text{s.t. } A_0 \xx_0 + B \xx_\A = \bm{b}    
                         \end{array}
  = \min_{\xx_\A \in\X_A} \phi_\A(\xx_\A) +  F(\xx_\A)
\end{align}
where
\begin{equation}
  \label{eq:Fdef}
  F(\xx_\A) \eqd   \min_{\xx_0 \in \X_0}  \max_{\llam} \phi_0(\xx_0) + \llam^\tr( A_0 \xx_0 + B \xx_\A - \bm{b}   ) \ .
\end{equation}
The function $F$ is not necessarily differentiable. However, it is subdifferentiable as it is convex. From there we can derive the following result on the DLMPs:
\begin{proposition} \label{prop:dlmpsSubgrad}
  For each node $n\in\N$, the DLMPs $\ddlmp_n,\ddlmq_n$ correspond to subgradients of the DSO's optimal cost with respect to $\pp_n$ (resp $\qq_n$), that is
  \begin{equation*}
\forall t \in \T, \      \dlmp\nti \in \partial_{p\nti}{F(\xx_\A)}  \ ,  \    \dlmq\nti  \in   \partial_{q\nti}{F(\xx_\A)}
  \end{equation*}
\end{proposition}
\begin{proof}
  We apply the sensitivity inequality \cite[sec. 5.57]{boyd2004convex} to the problem defined by $F(\xx_\A)$, parameterized by $\xx_\A$. This inequality  corresponds to the definition of the subgradients of $F(.)$.
\end{proof}
Because $\phi_0$ is not strictly convex w.r.t. $\xx_0$, in general we cannot guarantee the uniqueness of the DLMPS.

From \Cref{prop:dlmpsSubgrad}, one could wonder if DLMPs are to be always positive.
Indeed, a slight increase in active or reactive consumption at node $n$ would, in general, trigger the same increase (in addition to network losses) in production at the root node $0$ and, under the assumption that production costs $\phi_0$ is strictly increasing in $p^{p}_{0,t}$, then $\nabla_{p^{p}_{0,t}} \phi_0$ would be positive.
However, this is not always the case: the numerical example given in \Cref{sec:numericalStud} shows a case where a DLMP is slightly negative.
    
\section{Decentralized Coordination Methods}
\label{sec:coordMethods}

The idea behind the proposed coordination algorithm is to consider the OPF problem in a decentralized framework, where different parts of the set of decision variables are managed by different agents:
\begin{itemize}[wide]
  \item each LA $a\in\A$ decides of local consumption variables $\xx_a$ 
related to her affiliated nodes (and the corresponding consumption/production variables);
\item the DSO is responsible of the operation of the network, that is, of the variables $\xx_0$.%
\end{itemize}
Besides, a decentralized coordination mechanism is also relevant to address the partial information held by each agent. 
Typically, the network characteristics (topology, capacities, etc.) related to constraints \eqref{cons:OPF} are considered as confidential information by the DSO and shall not be revealed to other actors of the system. 
On the other hand, consumption and production constraints can constitute private information for electric consumers and, thus, should not be revealed by an aggregator to other actors or network operator.

 The essential step in the proposed DLMP-based procedure  is to rely on a decomposition method that enables to obtain $\xx_\A^*$ and DLMPs $\llam^*$ while respecting the decentralized structure of decisions and information: the DSO does not have access to $\X_\A$ and is responsible of variables $\xx_0$, while each LA $a$ does not have access to $\X_0$ or $\X_{\tilde{a}}$ for $ \tilde{a}\neq a $, and is responsible of variables $\xx_a$.
 In the remaining of this \Cref{sec:coordMethods}, we review different decomposition methods  to be used in this framework.

 To ensure the convergence and validity of the different methods, we consider the following additional standard assumption of strong duality:
    \begin{assumption} \label{assp:strongDuality}
A constraint qualification (e.g. Slater's condition) holds for \eqref{pb:formalComplete}, such that strong duality holds \cite[ch.~5]{boyd2004convex}:
      \begin{equation} \label{eq:strongDuality}
        \min_{\xx \in \X} \max_{\llam \in \rr^k} \L(\xx,\llam)  = \max_{\llam \in \rr^k}     \min_{\xx \in \X} \L(\xx,\llam) \ .
        \end{equation}
      \end{assumption}
  
\subsection{Dual Decomposition}

Considering $\xx_0$ and $\xx_\A$ satisfying \eqref{cs:formalx0} and \eqref{cs:formalxA}, the Lagrangian function associated to   \eqref{pb:formalComplete} is defined as:
\begin{equation} \label{eq:lagrangeanDef}
  \L(\xx_0,\xx_\A,\llam) \eqd \Phi(\xx) + \llam^\tr (A_0 \xx_0 + B \xx_\A - \bm{b}) \ ,
  \end{equation}
which is a basic ingredient of the following method. 

The dual decomposition method \cite{palomar2006tutorial} relies on the  consideration of the  subproblems:
\begin{align} \label{eq:P2relaxed}
\tag{$\pro'_a({\llam_a}) $}   \min_{\xx_a \in \X_a} \hh \phi_a(\xx_a) +\hh  \llam_a^\tr \hh B_a\xx_a  , \\
  \label{eq:P0relaxed} \tag{$\pro'_0({\llam}) $}  \min_{\xx \in \X_0 } \hh \phi_0(\xx_0)+ \llam^\tr A_0\xx_0 .%
\end{align}
A dual ascent enables to optimize  \eqref{pb:formalComplete} in a distributed way, by considering the following  \Cref{algo:dualAscent}:

\begin{algorithm}[H]%
\begin{algorithmic}[1]
\Require $\llam^{(0)}$, stopping criterion , steps $(\alpha_k)_k$ \; %
\State $k \leftarrow 0$ \;%
\While{ stopping criterion not true } %
\For{ each LA $a \in \A$}
\State [LA $a$] receive $\llam^{(k)}_a$ from DSO\;
\State [LA $a$]  $\xx_a^{(k+1)} \in  \sol \eqref{eq:P2relaxed}$ \;
\EndFor
\State [DSO]  $\xx_0^{(k+1)}  \in \sol \eqref{eq:P0relaxed} $\;
\State [DSO]  $\llam^{(k+1)} =  \llam^{(k)} + \alpha_k( A_0\xx_0^{(k+1)} + B\xx_{\A}^{(k+1)}\hh -\bm{b}) $ \;
\State $  k \leftarrow k+1 $ \;
\EndWhile
\end{algorithmic}
\caption{Decomposition through Dual Ascent}
\label{algo:dualAscent}
\end{algorithm}

If the sequence $(\alpha_k)_k $ is chosen such that $ \sum_k \alpha_k = + \infty$ while  $\sum_k \alpha_k^2 < \infty$, and $\Phi$ satisfies some strict convexity assumption, \Cref{algo:dualAscent} converges to a solution of \eqref{pb:formalComplete} \cite{palomar2006tutorial}, \cite[Ch.~6]{bertsekas1997nonlinear}.

\subsection{Alternating Direction Method of Multipliers}

The Alternating Direction Method of Multipliers (ADMM) was originally introduced in \cite{glowinski1975approximation} and gained in popularity due to applications to machine learning and the seminal survey \cite{boyd2011distributed}.

As for the Auxiliary Problem Principle, the method relies on the augmented Lagrangian function:
\begin{equation} \label{eq:augLagrangeanDef}
  \L_{\rho}(\xx_0,\xx_\A,\llam) \eqd \Phi(\xx) + \llam^\tr (A_0 \xx_0 + B \xx_\A - \bm{b}) + \frac{\rho}{2} \norm{A_0 \xx_0 + B \xx_\A - \bm{b}}_2^2 \ , 
\end{equation}
where $\rho>0$ is a parameter that is used as a step size in the method:
\begin{algorithm}[H]%
\begin{algorithmic}[1]
\Require $\llam^{(0)}$, stopping criterion , steps $(\alpha_k)_k$ \; %
\State $k \leftarrow 0$ \;%
\While{ stopping criterion not true } %
\For{ each LA $a \in \A$}
\State [LA $a$] receive $\llam^{(k)}_a$ and $A_0 \xx_0^{(k)}$ from DSO\;
\State [LA $a$]  $\xx_a^{(k+1)} \in  \argmin{\xx_a \in \X_a}  \  \phi_a(\xx_a) +\hh  \llam_a^{(k)\tr} \hh B_a\xx_a \hh +\txt\tfrac{\rho}{2} \norm{A_0 \xx_0^{(k)} + B \xx_\A - \bm{b}}_2^2 $\;
\EndFor
\State [DSO]  $\xx_0^{(k+1)}  \in \argmin{\xx_0 \in \X_0}  \  \phi_0(\xx_0) +\hh  \llam^{(k)\tr} \hh A_0\xx_0 \hh +\txt\tfrac{\rho}{2} \norm{A_0 \xx_0 + B \xx_\A^{(k+1)} - \bm{b}}_2^2 $\;
\State [DSO]  $\llam^{(k+1)} =  \llam^{(k)} + \rho( A_0\xx_0^{(k+1)} + B\xx_{\A}^{(k+1)}\hh -\bm{b}) $ \;
\State $  k \leftarrow k+1 $ \;
\EndWhile
\end{algorithmic}
\caption{Decentralized Optimization through ADMM}
\label{algo:ADMM}
\end{algorithm}
In ADMM, in addition to communicate the DLMPs $\llam^{(k)}$, the DSO also has to send the current profiles $A_0\xx_0^{(k)}$ to each LA at each iteration $k$.
Getting back to \eqref{eq:DSO-actPowBal}, this means that the DSO communicates a ``base'' profile $(\tilde{p}^{(k)}\nti, \tilde{q}^{(k)}\nti)\nti$ given by $\tilde{p}^{(k)}\nti \eqd -f\nti^{(k)} \hm +\displaystyle\hm \sum_{m: m_{-}=n}\hmmm \big(f\mti^{(k)} \hh- \hh \ell\mti^{(k)} R\mti \big)\hh -G_n v\nti^{(k)}$ , and similarly for $\tilde{q}^{(k)}\nti$.

The convergence of ADMM is ensured by the  following result:
\begin{theorem}{\cite[Sec.3.2.1]{boyd2011distributed}}\label{thm:admmCvg}
  Under \Cref{assp:strongDuality} and convexity of $\phi_0$ and $\phi_\A$, we have the following convergence result:
  \begin{itemize}
  \item feasibility convergence: $\norm{ A_0\xx_0^{(k)} + B\xx_{\A}^{(k)}\hh -\bm{b}} \underset{k \rightarrow \infty}{\longrightarrow} 0$ ,
  \item objective convergence: $\phi_0(\xx_0^{(k)}) + \phi_\A(\xx_{\A}^{(k)}) \underset{k \rightarrow \infty}{\longrightarrow} \Phi^*$ where $\Phi^*$ is the optimal value of \eqref{pb:formalComplete},
    \item dual (DLMPs) convergence: $\llam^{(k)} \ \underset{k \rightarrow \infty}{\longrightarrow} \llam^*$, where $\llam^*$ are optimal Lagrangian multipliers of  \eqref{pb:formalComplete}.
  \end{itemize}
\end{theorem}

\subsection{Primal-Dual Gauss-Seidel (PDGS) Iterations}

In  \Cref{algo:dualDecompAltMin} below, we propose a new method, referred to as PDGS, which relies on a Lagrangian relaxation of the coupling constraint \eqref{cs:formalCouplingLin} as the dual decomposition methods above.
The main difference is that we consider the original instances of \eqref{pb:DSO-OPF} instead of relaxed problem. For that, we rely on the resolution of the dual problem  of OPF problem \eqref{pb:DSO-OPF} to compute the DLMPs $\llam$, while computing the associated LAs decisions $\xx_\A$ afterwards.
Using the Lagrangian function $\L$ defined in \eqref{eq:lagrangeanDef} , problem \eqref{pb:formalComplete} can be written, with $\X \eqd \X_0 \times \X_\A$:
  \begin{align}
    \inf_{\xx \in \X} \sup_{\llam \in \rr^k} \L(\xx,\llam),
  \end{align}
  and the dual problem:
    \begin{equation} \label{pb:dualformal}
\sup_{\llam \in \rr^k}     \inf_{\xx \in \X} \L(\xx,\llam) \ = \sup_{\llam \in \rr^k} \psi_0(\llam, \xx_\A)  \ ,
    \end{equation}
    where we consider the \emph{(partial) dual function} $\psi_0$ defined as:
    \begin{equation*}
      \begin{split}
        \psi_0(\llam, \xx_\A) \eqd &  \min_{\xx_0 \in \X_0  } \big\{  \Phi (\xx_0,\xx_\A ) + \llam^\tr (A_0 \xx_0 + B \xx_\A - \bm{b}) \big\} \\
         =   \min_{\xx_0 \in \X_0  }   \big\{  \phi_0 & (\xx_0 ) + \llam^\tr A_0 \xx_0  \big\} +  \phi_\A (\xx_\A ) + \llam^\tr(B \xx_\A - \bm{b}) \ .
        \end{split}
      \end{equation*}
      Because $\X_0$ is a compact subset, $\psi_0$ is well defined. Using the notation of \eqref{pb:formalComplete}, the problem \eqref{pb:DSO-OPF} can be reformulated as:
      \begin{align*} 
  &      \min_{\xx_0\in\X_0}    \phi_0 (\xx_0 ) \\
 & \text{s.t. }        A_0 \xx_0  =\bm{b} - B \xx_\A  & (\llam)
      \end{align*}
where $\llam$ is the Lagrangian multiplier associated to the equality constraint, such that $\psi_0$ is the dual function of problem \eqref{pb:DSO-OPF} (translated by $\phi_\A(\xx_\A)$). Moreover, we have: 
      \begin{proposition}
        \label{prop:cvxconcaveDual}
 $\psi_0$  is concave in $\llam$ and convex in $\xx_\A$.
\end{proposition}
\ifreferonline
The proof is detailed in the online version \cite{paulinjacquot2020dlmp}.
\else
      \begin{proof}
        The function $\psi_0$ is the sum of the convex function $\xx_\A \mapsto \phi_\A(\xx_\A)$ and of an affine function of $\xx_\A$, thus it is convex in $\xx_\A$.
        As a minimum of concave functions of $\llam$ (because affine), it is concave in $\llam$.
      \end{proof}
      \fi
      The dual problem \eqref{pb:dualformal} is always feasible because $\X_0$ is nonempty and $\llam \in \rr^k$. However, it is not necessarily bounded: indeed, because strong duality holds, we know that the dual problem is unbounded \emph{iff} the primal problem  \eqref{pb:DSO-OPF} is infeasible \cite[Sec.5.2]{boyd2004convex}.
The \Cref{algo:dualDecompAltMin} presented below relies on dual solutions: to ensure its convergence, we rely on truncated dual problems, resulting in bounded dual solutions.
      \begin{proposition}
        \label{prop:boundedDual}
        The modified problem $\pro^K_0(\xx_a)$ with additional variable $\bm{u}_+, \bm{u}_-$:
        \begin{equation}
          \tag{$\pro^K_0(\xx_a)$}
          \begin{split}
  &      \min_{\xx_0\in\X_0, \ \bm{u}_+, \bm{u}_-\geq 0}    \phi_0 (\xx_0 ) + K (\bm{u}_{_+} - \bm{u}_{_-}) \\
  & \text{s.t. }        A_0 \xx_0  =\bm{b} - B \xx_\A + \bm{u}_{_+} - \bm{u}_{_-}  \quad \quad\quad\quad (\llam)
  \end{split}
      \end{equation}
      admits the same dual problem as \eqref{pb:DSO-OPF} with the additional constraint $K \dsone_{k} \leq \llam \leq K \dsone_{k}$.
\end{proposition}
      \begin{proof}
This comes from the definition of the dual problem. Details are omitted.
\end{proof}

\begin{algorithm}[ht]%
\begin{algorithmic}[1]
\Require $\llam^{(0)}$, stopping criterion \; %
\State $k \leftarrow 0$ \;%
\While{ stopping criterion not true } %
\For{ each LA $a \in \A$}
\State [LA $a$]$ \begin{array}{l}   \xx_\A \in  \argmin{\xx_a\in\X_a} \  { \phi_a(\xx_a)\hh  + \llam^{(k)}_a}^{\hh \tr} \hh B_a \xx_a  = \argmin{\xx_a\in\X_a} \  \psi_0( \llam^{(k)} \hh,\xx_a)  \end{array}$ \; \label{algline:dualDecomp:agg}
\State update solution $\xx_a^{(k+1)} \eqd \tfrac{1}{k+1} ( k \xx_a^{(k)}+ \xx_a)$ \;
\EndFor
\If {~$\pro_0(\xx_a^{(k+1)})$  is feasible} \label{algline:dualDecomp:DSOlamcond}
\State [DSO] obtain dual solution $\llam  \displaystyle\arg \max_{\llam \in \rr^k} \psi_0( \llam,\xx_\A^{(k+1)})$ of problem $\pro_0(\xx_a^{(k+1)})$\;
\Else
  \State  [DSO] obtain dual solution: $\llam \in  \displaystyle\arg \hh\hh\hh\max_{K \dsone_{k} \leq \llam \leq K \dsone_{k}} \psi_0( \llam,\xx_\A^{(k+1)})$ of modified problem $\pro^K_0(\xx_a^{(k+1)})$\;
  \EndIf \label{algline:dualDecomp:DSOlamend}
  \State [DSO]  update dual solution (DLMPs) $\llam^{(k+1)} \eqd \tfrac{1}{k+1} ( k \llam^{(k)}+ \llam)$ \;
\State $  k \leftarrow k+1 $ \;
\EndWhile \;
\end{algorithmic}
\caption{Decentralized Optimization through PDGS}
\label{algo:dualDecompAltMin}
\end{algorithm}

      The idea behind \Cref{algo:dualDecompAltMin} is that, as stated above, the primal problem \eqref{pb:DSO-OPF} is not always feasible (depending on the value of $\xx_\A$), but the dual problem is always feasible (although it can be unbounded) and we can always get a dual solution $\llam$. We rely on \Cref{prop:boundedDual} and solve the modified problem to ensure this dual solution remains bounded.
      We have  the following formal result:
      \begin{proposition}
        In case  of convergence, \Cref{algo:dualDecompAltMin} provides a DLMP-based decentralized coordination method which enables to optimize DERs while satisfying network constraints, giving an optimal solution of \eqref{pb:OPFc}.
      \end{proposition}
      \begin{proof}
        The algorithm is decentralized as $\phi_\A(\xx_\A)= \sum_{a\in\A}\phi_a(\xx_a)$ is separable and $B=diag(B_a)_{a\in\A}$ is block-diagonal, so that \Cref{algline:dualDecomp:agg} can be executed in a decentralized manner by each LA $a\in\A$, while the update of $\llam$ (\Cref{algline:dualDecomp:DSOlamcond} to \Cref{algline:dualDecomp:DSOlamend}) is executed independently by the DSO.
        
        Let $ \xx_\A^*, \llam^*$ be a fixed point of \Cref{algo:dualDecompAltMin}. Then $ (\xx_\A^*, \llam^*)$ satisfy the KKT conditions of \eqref{pb:operator_local}. Considering a primal solution $\xx_0^*$ for \eqref{pb:DSO-OPF} associated to $\llam^*$, then $(\xx_0^*,\llam^*)$ satisfies the KKT conditions of \eqref{pb:DSO-OPF}. The union of the two sets of conditions gives exactly the KKT conditions of  problem \eqref{pb:OPFc}, which shows that $ (\xx_0^*,\xx_\A^*, \llam^*) \in \sol \eqref{pb:OPFc}$.  
      \end{proof}
      
      There are two main advantages of \Cref{algo:dualDecompAltMin}:
      \begin{enumerate}
      \item we consider the actual problems on both sides:  LAs simply face the price incentives given by the DLMPs $\llam$ in their local optimization problem, while the DSO computes the network optimal power flow solution given the consumptions profiles on each node;
      \item the method ensures \emph{primal feasibility}: at the end of each iteration and as soon as \eqref{pb:DSO-OPF} is feasible, the DSO computes a feasible solution $\xx_0$ of $\pro_0(\xx_a^{(k+1)})$ satisfying $A\xx_0 + B\xx_\A^{(k+1)}=\bm{b}$. This is the main difference with Lagrangian methods such as ADMM where primal feasibility is only \emph{asymptotic} (\Cref{thm:admmCvg}).
        In practice, this is of main importance in our framework as the DSO needs a solution that is exactly feasible.
      \end{enumerate}
        
To help understand the convergence conditions of \Cref{algo:dualDecompAltMin}, it is relevant to consider a zero-sum game \cite{hofbauer2006best} interpretation:
\begin{proposition} \label{prop:algoValueGame}
        Consider the zero-sum game on $\psi_0$ where the first player minimizes $\psi_0$ on $\xx_\A \in \X_\A$, while the second player maximizes $\psi_0$ on $\llam \in \rr^k_K \eqd \{\llam \in \rr^k \ | \ \forall m , \ -K \leq \lambda_m \leq K\}$, for $K \geq 0$ large enough.
        Let $(\xx_\A^*,\llam^*)$  denote a saddle (equilibrium) point of this game.
        Then, there exists a primal solution $\xx_0^*$ of $\pro_0(\xx_a^*)$, and ($\xx_0^*, \xx_\A^*$)  defines a solution to central problem \eqref{pb:OPFc}.
      \end{proposition}
      \ifreferonline
      The proof is detailed in the online version \cite{paulinjacquot2020dlmp}.
      \else
      \begin{proof}
        Because of \Cref{prop:cvxconcaveDual}, we have convex-concave saddle function on convex and compact sets, thus the game has a value \cite{hofbauer2006best}. Then, the dual problem  $\max_{\llam \in \rr^k} \psi_0( \llam,\xx_\A^*)$ is bounded, and has a solution $\llam^*$.   Now suppose that we have chosen $K \geq 0$  such that $\llam^* \in \rr^K$. In that case, $\llam^*$ is a dual solution of both $\pro_0(\xx_a^*)$ and $\pro_0^K(\xx_a^*)$.
        We know that  there is a solution $\xx_0^*$ to the primal problem $\pro_0(\xx_a^*)$, associated to $\llam^*$, such that $A_0 \xx_0^* + B_\A \xx_\A^*- \bm{b} =0$. We have
        \begin{align}
          \Phi(\xx_0^*, &  \xx_\A^*) +0   = \L(\xx_0^*,  \xx_\A^*, \llam^*) \nonumber\\
           &= \max_{\llam \in \rr^k_K} \min_{\xx_\A\in\X_\A} \psi_0( \llam,\xx_\A) \nonumber\\
                      & = \max_{\llam \in \rr^k_K} \min_{\xx_0 \in \X_0,\xx_\A\in\X_\A}  \L(\xx_0,  \xx_\A, \llam) \label{eq:appliedStgDual} \\
                        &=  \min_{\xx_0 \in \X_0,\xx_\A\in\X_\A} \max_{\llam \in \rr^k_K} \L(\xx_0,  \xx_\A, \llam) \nonumber\\
           &=  \min_{ \substack{\xx_0 \in \X_0,\xx_\A\in\X_\A \\ A_0 \xx_0^* + B_\A \xx_\A^* = \bm{b}}} \Phi(\xx)                 \nonumber   
        \end{align}
        where \eqref{eq:appliedStgDual} follows from \Cref{assp:strongDuality}.
        \end{proof}
\fi

The interpretation based on a zero-sum game goes further, as \Cref{algo:dualDecompAltMin}  actually implements the so-called  \emph{best response} dynamics (BRD). Thus, we are able to show that the method converges, as given in \Cref{thm:GScvg} below.
\begin{theorem}
  \label{thm:GScvg}
  For $K\geq 0$ large enough, and under \Cref{assp:exiSsolution,assp:strongDuality}, the sequences $(\llam^{(k)})_k$ and $(\xx_\A^{(k)})_k$ generated by \Cref{algo:dualDecompAltMin} converge respectively to a dual and a (partial) primal solutions of central problem \eqref{pb:OPFc}.
\end{theorem}
\begin{proof}
  Because of \Cref{prop:cvxconcaveDual}, $\psi_0$ is  a convex-concave saddle function on convex and compact strategy sets $\X_0$ and $\rr^k_K$. Thus, we can apply the convergence of best-response in discrete vanishing stepsizes \cite[Prop.7]{hofbauer2006best} which ensures that the sequences converge to an equilibrium point.
We can conclude with \Cref{prop:algoValueGame}.
\end{proof}
In \Cref{sec:numericalStud}, we give a numerical example of  the convergence of \Cref{algo:dualDecompAltMin} and compare it to ADMM. One can observe that, in the example considered here, the convergence of PDGS is quite slow.

\section{Decentralized Coordination Mechanism}
\label{sec:mechanisms}

The decentralized coordination mechanism that we propose in this paper is the following:
\newcommand{\taxPen}{\tau^{\text{pen}}}
\newcommand{\paym}{\mathfrak{P}}
\makeatletter
\renewcommand{\ALG@name}{Procedure}
\makeatother
\begin{algorithm}[H]%
\begin{algorithmic}[1]
\State  \emph{computation step: } an optimization decomposition method is run to coordinate LAs and the DSO, during which each LA $a \in \A$ sends a sequence of profiles $(\xx_a^{(k)})_{k=1}^{K}$ until convergence to  decisions  $\xx_a^*$, associated to  Lagrangian multipliers (DLMPs) $\llam^*$ for the DSO,  corresponding to an optimal primal-dual solution $(\xx_0^*,\xx_\A^*,\llam^*)$ of \eqref{pb:formalComplete}. In particular, each aggregator agrees on realizing the announced profile $\xx^*_a$;
\State \emph{realization step: } each LA $a$  realizes the profile $\hat{\xx}_a$, which can be measured by the DSO. %
\State \emph{penalization step: }  if $\hat{\xx}_a \neq \xx^*_a$ for at least one $a \in \A$, then the DSO recomputes DLMPs as the dual solution $\hat{\llam}$ of problem $\pro_0(\hat{\xx}_\A)$, and charges each agent $a$ with payments $ \paym_a \eqd (\hat{\llam} \cdot \hat{\xx}_a)$ and penalize cheating LAs with a tax $\taxPen$.
 \end{algorithmic}
\caption{DLMP-based Coordination Procedure}
\label{proc:coordination}
\end{algorithm}

\makeatletter
\renewcommand{\ALG@name}{Algo.}
\makeatother

\subsection{Mechanism Design Discussion}

Because we rely on a decentralized decomposition method, the exchange of information is limited: in particular, aggregators do not have to send all their information $(E_n,\lbxx_n,\ubxx_n)_{n\in\N_a}$ as considered in some mechanisms as in \cite{samadi2012advanced}.
Yet, Procedure \ref{proc:coordination} describes a formal mechanism, with each agent (aggregator)  announcing the sequence  $(\xx_a^{(k)})_{k=1}^{K}$, then realizing $\hat{\xx}_a$ and being charged $ \paym_a \eqd \hat{\llam} \cdot \hat{\xx}_a$.

As  $(\xx_0^*,\xx_\A^*,\llam^*)$ constitutes an optimal solution of \eqref{pb:formalComplete}, it is \emph{optimal} for each LA $a$ to realize the profile $\xx^*_a$ under the incentives $\llam^*$.

\medskip

However,  aggregators are still required to communicate the sequence of consumption profiles $(\xx_a^{(k)})_{k=1}^{K}$. From the point of view of mechanism design \cite[Part II]{nisan2007algorithmic}, one can think that aggregators may be tempted to send profiles ${\tilde{\xx}_a}$ that do not comply with the update rule of the chosen decomposition algorithm, in order to manipulate the DLMPs $\llam^*$ computed in step 2.
With LA $a$ being \emph{truthful}, its final ``agreed'' profile $\xx^*_a$ should always satisfy (up to the convergence error):
\begin{align*}
 & \xx_a \in \argmin{\xx_a\in\X_a} \  \phi_a(\xx_a) + \llam^*_a B_a \xx_a \\
\text{ with } (\llam^*,\xx_0^*)  \text{ s.t. } & \xx_0^* \in \argmin{\xx_0 \in \X_0} \ \phi_0(\xx_0)  + \llam^* A \xx_0.
\end{align*}

Because the operator is able to verify (step 3) if the realized profile $\hat{\xx_a}$ corresponds to the agreed profile $\tilde{\xx}_a$, the operator can prevent deviations of aggregators on step 3 by imposing a large tax  $\taxPen$ on cheating aggregators.

\subsection{A DLMP-based coordination Procedure is not Incentive-Compatible}

Preventing the LAs to deviate in step 2  is not sufficient to ensure the truthfulness of the mechanism: a cheating aggregator could also respond untruthfully during the computation step with a sequence of profiles $(\tilde{\xx}_a^{(k)})_{k=1}^{K}$ with  $ \tilde{\xx}_a^{(K)} \eqd {\tilde{\xx}_a}$.
More precisely, if an LA want to cheat,  it can respond to each iteration DLMPs $\tilde{\llam}^{(k)}$ by acting with a different cost function $\tilde{\phi_a} \neq \phi_a$ and/or feasible set $\tilde{\X}_a \neq \X_a$ (and such that $\tilde{\X}_a \subset \X_a$ in order to converge to a feasible profile for the LA, that can be realized in step 2).

In that way, the algorithm would still converge to $\tilde{\llam}, \tilde{\xx_0} $ satisfying $$\tilde{\xx}_0 \in \argmin{\xx_0 \in \X_0} \ \phi_0(\xx_0)  + \tilde{\llam} A \xx_0.$$

If we assume that this cheating aggregator would not deviate from $ {\tilde{\xx}_a}$ in the realization step (because the tax  $\taxPen$ is large enough to discourage her), then the  aggregator would benefit from cheating if $ \phi_a(\tilde{\xx}_a) + \tilde{\llam}_a B_a \tilde{\xx}_a  < \phi_a(\xx^*_a) + \llam^*_a B_a \xx^*_a $.
The following counter-example shows an instance where this strict inequality holds, proving that, in general, the \emph{DLMP-based mechanism is not incentive compatible.}

\begin{example} \label{ex:counterexampleIncentive} Let us consider the toy example with only one node (one aggregator) linked to the root node 0 so that $\N= \{0,1 \}$, and two time periods $\T=\{0,1 \}$. Parameters and topology are given in \Cref{fig:toyCheating}.
  The operator has an energy cost of $$\phi_0(\pp_0) =(10 |p_{01}| + 10 p_{01}^2) +  (3 |p_{02}| + 2 p_{01}^2),$$
  while the LA has a \emph{preferred profile} $\pp_1^{\sharp}\eqd (1.5,1.5)$ defining the cost $$\phi_1(\pp_1)\eqd \omega \norm{\pp_1- \pp^\sharp_1}_2^2  \text{  where  } \omega \eqd 10. $$
  We assume that the LA has an\emph{ actual} upper bound on the admissible power on time period $t=1$ given by $\overline{P}_{11}=1.5$.
  If the LA announces its true upper bound of 1.5, she will end up with  a utility $\phi_a=  -33.57 $  and a total cost of  $\phi_a + \paym_a = -15,13$.
  When cheating and announcing a \emph{lower} upper bound $\tilde{\overline{P}}=1.0$, her cost will be $ \tilde{\phi}_a+ \tilde{\paym}_a= -32.49$ and total cost of $-15.47$, lower than the total cost when being truthfully.
The two solutions are illustrated on \Cref{fig:sol-PubAnnounced} below.
  \begin{figure}[h]
    \centering
    \hspace{-1cm}
  \begin{scriptsize}\begin{tikzpicture}[scale=1]
\node [draw,circle] (node0) at (0,0) {0} ;
\node [left= 0.5cm of node0] (res0) {\begin{minipage}{2.7cm}
$G_0= 0$ \\ 
$B_0= 0$ \\ 
$\bm{V}_{0}=1.0$ (fix)
\end{minipage}} ;
\node [draw,circle] (node1) at (0.0, -3.5) {1} ;
 \draw [black ,line width=2pt] (node0) -- (node1) node [pos=0.5,sloped,above,black] { $R,X=0.001,0.12$ } node [pos=0.5,sloped,below,black] {$S = 5$ } ; 
  
\node [left= 1 cm of node1, yshift=1cm] (res1) {\fbox{\begin{minipage}{4cm}
$\underline{\bm{P}}_{1},\overline{\bm{P}}_{1}=[0.3, 0.2],[1.5, 2.0]$ \\ 
$\underline{\bm{Q}}_{1}, \overline{\bm{Q}}_{1}=[-0.5, -0.5],[1, 1]$ \\ 
$\underline{{V}}_{1}, \overline{{V}}_{1}=0.7, 1.3$ \\ 
$G_{1}= 0$ \ , \ $B_{1}= 0.0011$ \\ 
$E_{1}= 1$ \ , \ $\eta_{1}= 0.3$ \\  
\end{minipage}} };
\end{tikzpicture}
\end{scriptsize}
\hspace{0.5cm}
\includegraphics[width=0.6\linewidth]{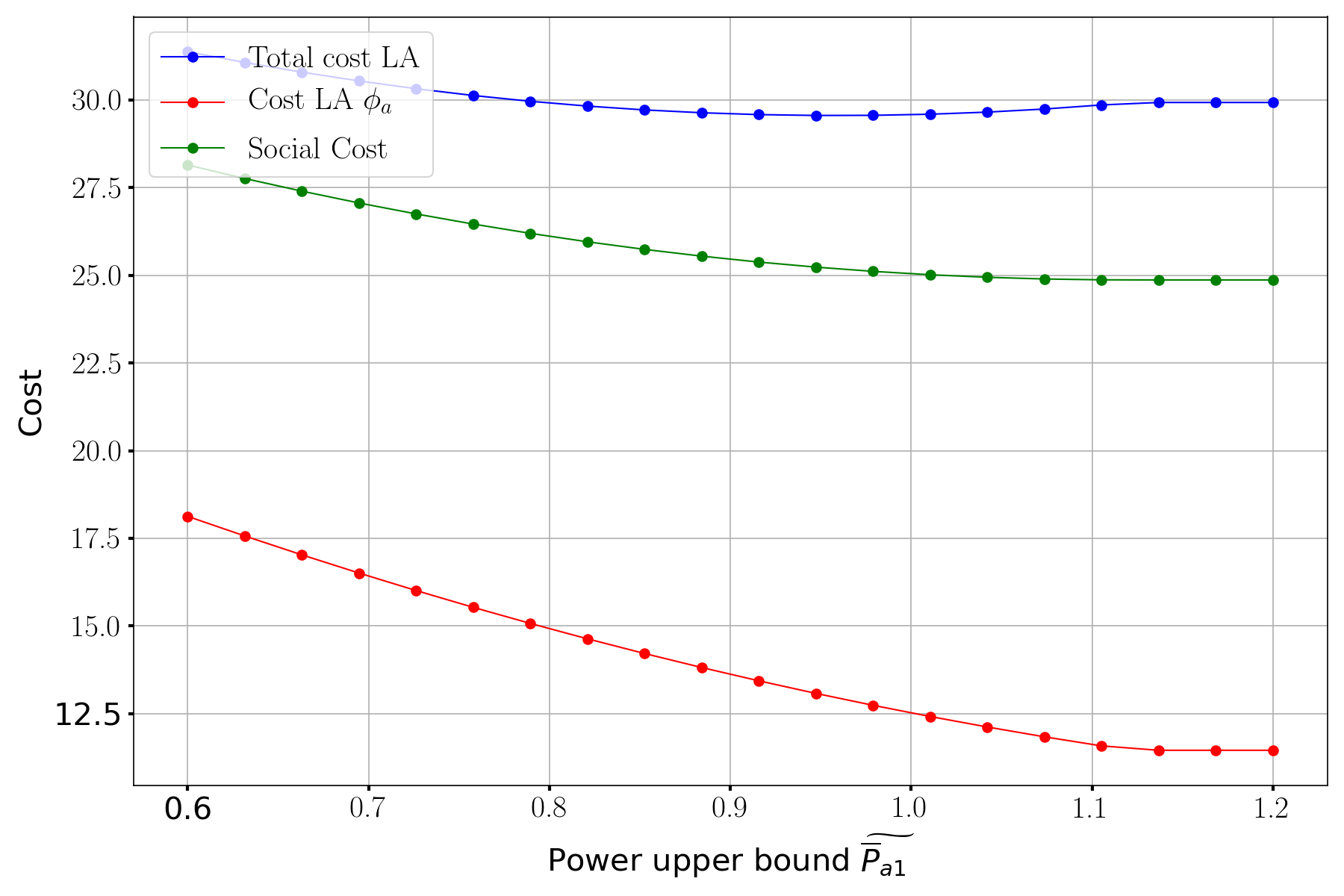} 
    \caption{A toy example where the aggregator benefits from cheating}
    \label{fig:toyCheating}
  \end{figure}
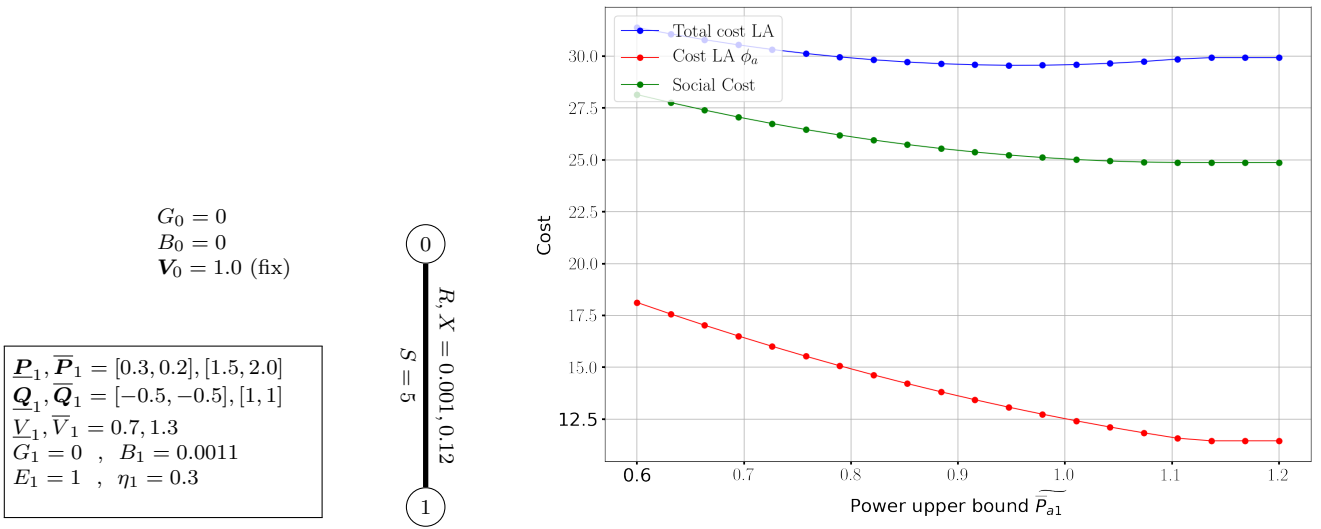

  \begin{figure}[h]
\begin{subfigure}{0.49\columnwidth} \centering
   \fbox{\begin{minipage}{\textwidth} \includegraphics[width=0.6\linewidth]{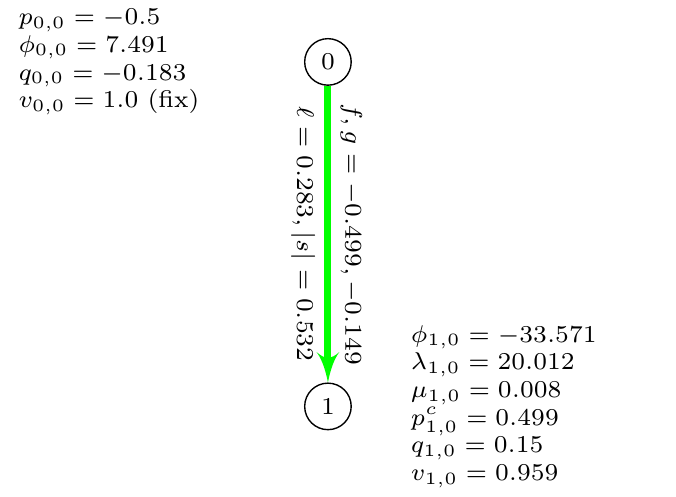}
 \hspace{-2cm}    \includegraphics[width=0.6\linewidth]{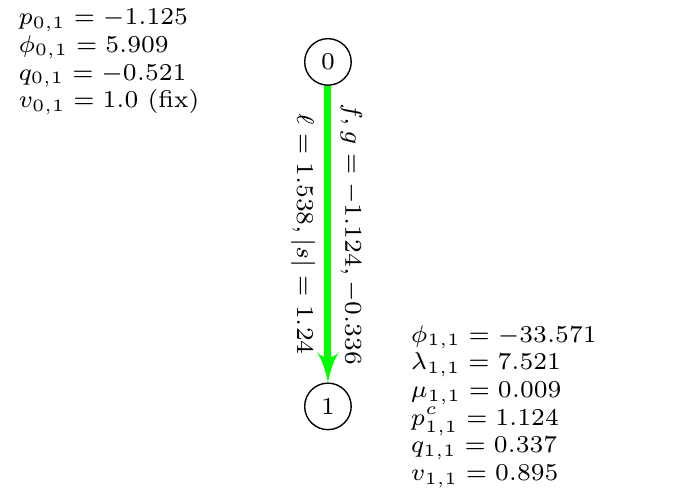}\end{minipage}}
 \subcaption{ $\tilde{\overline{P}}_{11}=1.5$}
\end{subfigure}
\begin{subfigure}{0.49\columnwidth} \centering
   \fbox{\begin{minipage}{\textwidth} 
       \includegraphics[width=0.6\linewidth]{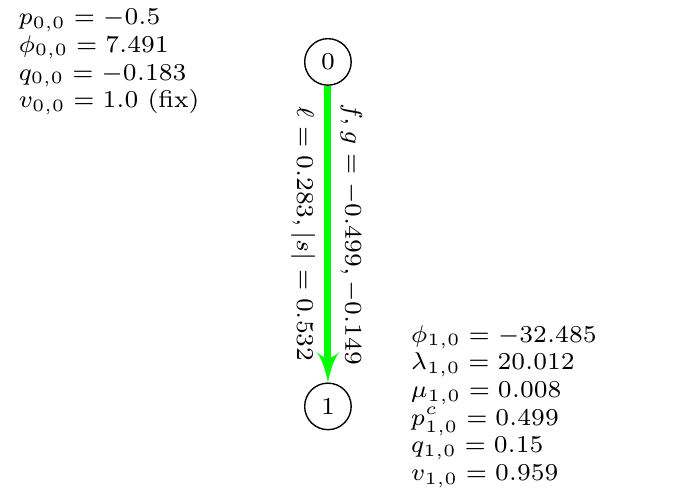}
 \hspace{-2cm}    \includegraphics[width=0.6\linewidth]{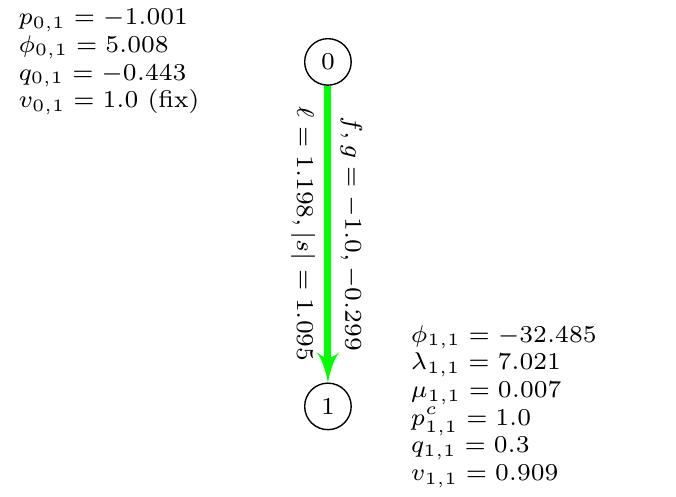} \end{minipage}}
\subcaption{ $\tilde{\overline{P}}_{11}=1.0$ }
\end{subfigure}
 \caption{Solution with Announced Power Upper bound : actual and cheated}
 \label{fig:sol-PubAnnounced}
  \end{figure}

\end{example}

  \medskip

  \subsection{Comparison with VCG mechanism}
\label{sec:vcg-meca}

A standard mechanism that achieves several desirable properties is the so-called Vickrey-Clarkes-Groves (VCG) mechanism \cite[Sec.~10.6]{shoham2008multiagent}.

The framework considered in this paper is slightly different from the standard framework considered in mechanism design, because we have an additional entity, the DSO, that is not considered as an agent in the mechanism, but whose cost $\phi_0(\xx_0)$ has to be considered in the objective.
Indeed, instead of minimizing the social cost $\sum_{a\in\A} \phi_a(\xx_a)$, we minimize $\Phi(\xx_0,\xx_\A)=\phi_0(\xx_0) + \sum_{a\in\A} \phi_a(\xx_a)$.
Because we just added the term $\phi_0(\xx_0)$ to the sum of agents' cost functions, the social choice function considered here is in the class of \emph{affine maximizers} \cite[Def.~10.5.4]{shoham2008multiagent}.
\newcommand{\info}{\mathbb{I}}
Let us denote by $\info_\A \eqd (\phi_a, \X_a)_{a\in\A} $ the information transmitted by agents to the DSO. The VCG mechanism would  compute:
\begin{align}
  & \xx(\info_\A) \eqd ( \xx_0^* , \xx_\A^*) \in \arg \min \{ \phi_0(\xx_0) + \sum_{a\in\A} \phi_a(\xx_a) \ | \   (\ref{cons:OPF}) \ \& \ \forall a,  \xx_a \in \X_a \} \\
  & \forall a\in\A,  \paym_a \eqd -  \tau_a^{c}(\info_{-a})    + \Big[ \phi_0\big(\xx_0(\info_\A)\big) + \sum_{a' \neq a } \phi_{a'}\big(\xx(\info_\A)\big) \Big] \ ,
  \end{align}
  where $\paym_a$ is the payment made by $a$ to the operator, $\info_{-a} \eqd (\phi_{a'}, \X_{a'})_{a' \neq a}$ and $\tau_a^{c}(\info_{-a})$ refers to the so-called Clarke tax:
  \begin{align} \label{eq:clarktax}
&     \tau_a^{c}(\info_{-a}) \eqd \phi_0\big(\xx_0(\info_{-a} ) \big)  + \sum_{a' \neq a } \phi_{a'}\big(\xx_{a'}(\info_{-a})\big) \\
    \text{ where } & \big(\xx_0(\info_{-a}),\xx_{-a}(\info_{-a}) \big) \in \arg \min_{\xx_0, \xx_{-a}}  \phi_0(\xx_0)  + \sum_{a' \neq a } \phi_{a'}(\xx_{a'}).
  \end{align}
  In our framework, the solutions $\big(\xx_0(\info_{-a}),\xx_{-a}(\info_{-a}) \big)$ solved the OPF problem without considering the loads of LA $a$, that is, having $\xx_a=0$, but conserving the same network structure, as the network variables are controlled by the DSO.
  In the toy example of \Cref{fig:toyCheating} above, the VCG payment of the LA $1$ is $\paym_1 = \phi_0(\xx_0^*)$, because there is only one LA in the network, and the DSO costs would be zero without it.
  
The standard VCG mechanism is not decentralized: a priori, each agent $a$ has to provide its complete information $\info_a \eqd (\phi_a, \X_a)$ to the DSO.
One can wonder if it is possible to reduce the amount of information provided, while keeping the structure of the VCG mechanism: indeed, using a decomposition algorithm, each agent $a$ would only have to provide:
\begin{itemize}
\item  the profile $\xx_a^*$ (more precisely a sequence of profile responses $(\xx_a^{*(s)})_s$ for an iterative algorithm as in \Cref{sec:coordMethods}, but the ultimate iteration, on which the algorithm converges, is the most important)
\item  the values $\phi_a\big(\xx(\info_\A) \big) $ and $\phi_a\big(\xx(\info_{\A \setminus \{a'\}})\big) $ for each $a'\neq a$, to compute payment $\pi_{a'}$ for each $a'$).
  \end{itemize}
Let us refer to this modified mechanism as Decentralized-VCG (DVCG). As the standard VCG, DVCG provides a mechanism that is efficient (in the sense that it provides an optimal solution of problem \eqref{pb:OPFc}), and incentive-compatible: 
  \begin{proposition}
    Under the VCG mechanism, responding optimally to the DSO and announcing the true profile $\xx_a^*$ and true values $\phi_a\big(\xx(\info_\A) \big) , \phi_a\big(\xx(\info_{\A \setminus \{a'\}})\big)$ is a dominant strategy for each aggregator $a$.
  \end{proposition}
  \begin{proof}
    As the cost values $\phi_a\big(\xx(\info_\A) \big) , \phi_a\big(\xx(\info_{\A \setminus \{a'\}})\big)$ provided by $a$ do not intervene in the payment $\paym_a$, $a$ is indifferent to the values it provides, and has no interest in providing false values.
    Now imagine that $a$ would lie during the decomposition algorithm, such that the final profiles obtained after convergence are $\hat{\xx} \neq \xx^*$.
    Then the final cost of $a$ would be $\phi_a(\hat{\xx}_a) + \paym_a = \Phi(\hat{\xx}) -   \tau_a^{c}(\info_{-a}) \geq \Phi(\xx^*) - \tau_a^{c}(\info_{-a})$.
    Thus, $a$ is better off announcing the truth, making the decomposition algorithm converging to the optimal solution $\xx^*$.
  \end{proof}

  However, DVCG (and furthermore VCG) would be impractical to implement in practice.
  A first problem is that, even if each agent is not impacted by the payments made by other agents, it may not be indifferent to it: in practical cases, if we consider several aggregators that are competing, each of them may have interest to provide false information to make payments of competitors larger.
  A second disadvantage is that, to compute Clarke taxes  \eqref{eq:clarktax}, the operator will have to solve $|\A|$ additional decentralized problems.
  
\section{Numerical Example}
\label{sec:numericalStud}

\newif\iftoymodel
\toymodelfalse

\iftoymodel
\input{toymodel.tex}
\else
\fi

In this example, we consider the 15 buses network  proposed by Papavasiliou \cite{papavasiliou2017DLMP}, but with flexible active and reactive loads instead of fixed ones, and we consider a time set $\T \eqd\{0, 1\}$, of 2 time periods.
\ifreferonline
The parameters $(R_n,X_n,S_n,B_n,V_n)$ are those of \cite{papavasiliou2017DLMP}, while parameters  $(\bm{\uP}_n,\bm{\oP}_n,E_n,\bm{\uQ}_n,\bm{\oQ}_n,\tauC_n)_n$  are generated randomly based on the values \cite{papavasiliou2017DLMP}. We refer the reader to \cite{paulinjacquot2020dlmp} for details.
\else
The network structure can be observed on \Cref{fig:resFlows}.

For each bus $n$, the parameters $R_n,X_n,S_n,B_n$ are those of \cite{papavasiliou2017DLMP}, given in \Cref{tab:network14param}, while parameters $(\bm{\uP}_n,\bm{\oP}_n,E_n,\bm{\uQ}_n,\bm{\oQ}_n,\tauC_n)_n$  are generated as follows. For each  $t\in\T$, and with $\hat{p}_n^c, \hat{q}_n^c$ denoting the fixed active and reactive load values considered in \cite{papavasiliou2017DLMP}:
 \begin{itemize}[wide]
 \item $\uP\nt$ is chosen randomly as $\uP\nt \sim \mc{U}([0,\hat{p}_n^c])$, where $\mc{U}(I)$ denotes the uniform distribution on $I$;
 \item $\oP\nt$ is chosen randomly as $\oP\nt \sim \mc{U}([\hat{p}_n^c, 2 \hat{p}_n^c])$;
 \item $\uQ\nt, \oQ\nt$ are chosen similarly considering $\hat{q}_n^c$;
 \item $E_n$ is chosen as $E_n \sim \mc{U}([\sum_t \uP\nt, \sum_t \oP\nt])$;
   \item $\tauC_n$ is fixed as $\tauC_n \eqd \hat{q}_n^c / \hat{p}_n^c$.
   \end{itemize}
   In this example, we consider that LAs are indifferent between  consumption profiles as long as they are feasible, that is, $\phi_a=0$ for each $a\in\A$.
Following \cite{papavasiliou2017DLMP}, we consider that only the bus $11$ has a renewable production, with $ \upRE_{11,t} \sim \mc{U}([0,0.6])$ and $\rlbRE \eqd 0$ (the renewable production is fully active). The bounds $(\uV_n,\oV_n)$ are taken to 0.81 and $1.21$ for each  $n\in\N$, while $V_0=1.0$.

We consider the objective $ \Phi(\xx)= \phi_0(\xx_0)=  \sum_{t\in\T}c_t( \pRE_{0t}) $, with cost functions chosen as follows: time period 0 has an expensive price given by $c_0:p \mapsto p + p^2$, while time period 1 has a cheaper price given by  $c_1 : p \mapsto p$.

\begin{table}[H]
  \centering
 \begin{small}
 \begin{tabular}{c||c|c|c|c|c|c}
 $n$ & $R_n$ & $X_n$ & $S_n$ & $\mathbb{E}[p\nt]$ & $\mathbb{E}[q\nt]$ & $B_n \hh\cdot\hh 10^{3}$\\ \hline
0  & 0 & 0 & 0 & 0 & 0 & 0\\
1  & 0.001 & 0.12 & 2 & 0.7936 & 0.1855 & 1.1\\
2  & 0.0883 & 0.1262 & 0.256 & 0 & 0 & 2.8\\
3  & 0.1384 & 0.1978 & 0.256 & 0.0201 & 0.0084 & 2.4\\
4  & 0.0191 & 0.0273 & 0.256 & 0.0173 & 0.0043 & 0.4\\
5  & 0.0175 & 0.0251 & 0.256 & 0.0291 & 0.0073 & 0.8\\
6  & 0.0482 & 0.0689 & 0.256 & 0.0219 & 0.0055 & 0.6\\
7  & 0.0523 & 0.0747 & 0.256 & -0.1969 & 0.000 & 0.6\\
8  & 0.0407 & 0.0582 & 0.256 & 0.0235 & 0.0059 & 1.2\\
9  & 0.01 & 0.0143 & 0.256 & 0.0229 & 0.0142 & 0.4\\
10 & 0.0241 & 0.0345 & 0.256 & 0.0217 & 0.0065 & 0.4\\
11 & 0.0103 & 0.0148 & 0.256 & 0.0132 & 0.0033 & 0.1\\
12  & 0.001 & 0.12 & 1 & 0.6219 & 0.1291 & 0.1\\
13  & 0.1559 & 0.1119 & 0.204 & 0.0014 & 0.0008 & 0.2\\
14  & 0.0953 & 0.0684 & 0.204 & 0.0224 & 0.0083 & 0.1
 \end{tabular}
\end{small}
\caption{Parameters for the 15 buses network \cite{papavasiliou2017DLMP}}
\label{tab:network14param}
\end{table}

\fi

  \ifreferonline
The  optimal values of all variables and the optimal power flows are detailed in the online version \cite{paulinjacquot2020dlmp}.
  \else

  The SOCP problem \eqref{pb:OPFc} is solved with the CvxOpt Python library \cite{andersen2013cvxopt} in 0.53s on a laptop with a processor of 2.6GHz. The solutions obtained are detailed in \Cref{tab:resNum}, while \Cref{fig:resFlows} shows the active flows directions and the saturated lines.
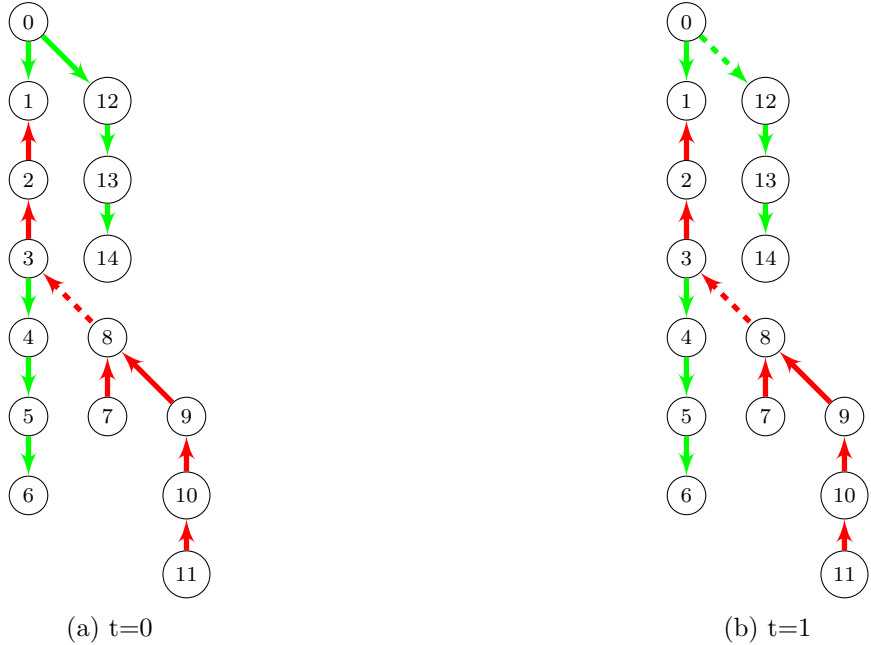
\begin{figure}[h]
\begin{subfigure}{0.49\columnwidth} \centering
 \begin{scriptsize}
\begin{tikzpicture}[scale=0.3]
\node [draw,circle] (node0) at (0,0) {0} ;
% \node [left= 0.5cm of node0] (res0) {\begin{minipage}{2.7cm}
% $p_{0,0}=-0.56$ \\ 
% $\mathcal{C}_0= c_0(p_{0,0})=0.873$ \\ 
% $\text{SC}_0(\sum_n p_{n,0})=0.843$ \\ 
% $q_{0,0}=-0.249$ \\ 
% $v_{0,0}=1$ (fix)
% \end{minipage}} ;
\node [draw,circle] (node1) at (0.0, -3.5) {1} ;
 \draw [-latex',green,line width=2pt] (node0) -- (node1); 
  
\node [draw,circle] (node2) at (0.0, -7.0) {2} ;
 \draw [latex'-, red,line width=2pt] (node1) -- (node2); 
  
\node [draw,circle] (node3) at (0.0, -10.5) {3} ;
 \draw [latex'-, red,line width=2pt] (node2) -- (node3); 
  
\node [draw,circle] (node4) at (0.0, -14.0) {4} ;
 \draw [-latex',green,line width=2pt] (node3) -- (node4); 
  
\node [draw,circle] (node5) at (0.0, -17.5) {5} ;
 \draw [-latex',green,line width=2pt] (node4) -- (node5); 
  
\node [draw,circle] (node6) at (0.0, -21.0) {6} ;
 \draw [-latex',green,line width=2pt] (node5) -- (node6); 
  
\node [draw,circle] (node8) at (3.5, -14.0) {8} ;
 \draw [latex'-, red, style = dashed ,line width=2pt] (node3) -- (node8); 
  
\node [draw,circle] (node7) at (3.5, -17.5) {7} ;
 \draw [latex'-, red,line width=2pt] (node8) -- (node7); 
  
\node [draw,circle] (node9) at (7.0, -17.5) {9} ;
 \draw [latex'-, red,line width=2pt] (node8) -- (node9); 
  
\node [draw,circle] (node10) at (7.0, -21.0) {10} ;
 \draw [latex'-, red,line width=2pt] (node9) -- (node10); 
  
\node [draw,circle] (node11) at (7.0, -24.5) {11} ;
 \draw [latex'-, red,line width=2pt] (node10) -- (node11); 
  
\node [draw,circle] (node12) at (3.5, -3.5) {12} ;
 \draw [-latex',green,line width=2pt] (node0) -- (node12); 
  
\node [draw,circle] (node13) at (3.5, -7.0) {13} ;
 \draw [-latex',green,line width=2pt] (node12) -- (node13); 
  
\node [draw,circle] (node14) at (3.5, -10.5) {14} ;
 \draw [-latex',green,line width=2pt] (node13) -- (node14); 
  
\end{tikzpicture}

\end{scriptsize}

%%% Local Variables:
%%% mode: latex
%%% TeX-master: "network_ieee"
%%% End:
 \subcaption{ t=0}
\end{subfigure}
\begin{subfigure}{0.49\columnwidth} \centering
\begin{scriptsize}
\begin{tikzpicture}[scale=0.3]
\node [draw,circle] (node0) at (0,0) {0} ;
%\node [left= 0.5cm of node0] (res0)  {\begin{minipage}{2.7cm}
% $p_{0,1}=-1.299$ \\ 
% $\mathcal{C}_1= c_1(p_{0,1})=1.299$ \\ 
% $\text{SC}_1(\sum_n p_{n,1})=1.287$ \\ 
% $q_{0,1}=-0.549$ \\ 
% $v_{0,1}=1$ (fix)
% \end{minipage}}
%;
\node [draw,circle] (node1) at (0.0, -3.5) {1} ;
 \draw [-latex',green,line width=2pt] (node0) -- (node1); 
  
\node [draw,circle] (node2) at (0.0, -7.0) {2} ;
 \draw [latex'-, red,line width=2pt] (node1) -- (node2); 
  
\node [draw,circle] (node3) at (0.0, -10.5) {3} ;
 \draw [latex'-, red,line width=2pt] (node2) -- (node3); 
  
\node [draw,circle] (node4) at (0.0, -14.0) {4} ;
 \draw [-latex',green,line width=2pt] (node3) -- (node4); 
  
\node [draw,circle] (node5) at (0.0, -17.5) {5} ;
 \draw [-latex',green,line width=2pt] (node4) -- (node5); 
  
\node [draw,circle] (node6) at (0.0, -21.0) {6} ;
 \draw [-latex',green,line width=2pt] (node5) -- (node6); 
  
\node [draw,circle] (node8) at (3.5, -14.0) {8} ;
 \draw [latex'-, red, style = dashed ,line width=2pt] (node3) -- (node8); 
  
\node [draw,circle] (node7) at (3.5, -17.5) {7} ;
 \draw [latex'-, red,line width=2pt] (node8) -- (node7); 
  
\node [draw,circle] (node9) at (7.0, -17.5) {9} ;
 \draw [latex'-, red,line width=2pt] (node8) -- (node9); 
  
\node [draw,circle] (node10) at (7.0, -21.0) {10} ;
 \draw [latex'-, red,line width=2pt] (node9) -- (node10); 
  
\node [draw,circle] (node11) at (7.0, -24.5) {11} ;
 \draw [latex'-, red,line width=2pt] (node10) -- (node11); 
  
\node [draw,circle] (node12) at (3.5, -3.5) {12} ;
 \draw [-latex',green, style = dashed ,line width=2pt] (node0) -- (node12); 
  
\node [draw,circle] (node13) at (3.5, -7.0) {13} ;
 \draw [-latex',green,line width=2pt] (node12) -- (node13); 
  
\node [draw,circle] (node14) at (3.5, -10.5) {14} ;
 \draw [-latex',green,line width=2pt] (node13) -- (node14); 
  
\end{tikzpicture}

\end{scriptsize}

%%% Local Variables:
%%% mode: latex
%%% TeX-master: "network_ieee"
%%% End:
\subcaption{ t=1 }
\end{subfigure}
 \caption{Directions of the active flows $\bm{f}$  at the optimal solution. \textit{Saturated lines are dashed.}}
 \label{fig:resFlows}
 \end{figure}

 \begin{table}[h]
   \centering
    \begin{scriptsize}
  \begin{tabular}{c|c|c|c|c|c|c|c|c}
$n$&  $\dlmp_{n,0}$ &  $\dlmq_{n,0}$ &  $p_{n,0}^c$ &  $q_{n,0}$ &  $f_{n,0}$ & $ g_{n,0}$ & $\ell_{n,0}$ &  $v_{n,0}$  \\\hline
0 & 2.12 & 0.0 & -0.56 & -0.249 & - & - & - & 1 \\
1 & 2.122 & 0.008 & 0.623 & 0.146 & -0.427 & -0.188 & 0.229 & 0.951 \\
2 & 2.049 & 0.029 & 0.0 & 0.0 & 0.199 & -0.038 & 0.042 & 0.975 \\
3 & 1.937 & 0.055 & 0.028 & 0.012 & 0.205 & -0.032 & 0.042 & 1.017 \\
4 & 1.939 & 0.055 & 0.005 & 0.001 & -0.019 & -0.003 & 0.0 & 1.016 \\
5 & 1.94 & 0.055 & 0.001 & 0.0 & -0.014 & -0.002 & 0.0 & 1.016 \\
6 & 1.942 & 0.055 & 0.013 & 0.003 & -0.013 & -0.003 & 0.0 & 1.014 \\
8 & 0.003 & 0.177 & 0.023 & 0.006 & 0.256 & -0.016 & 0.063 & 1.036 \\
7 & 0.0 & 0.177 & -0.131 & 0.0 & 0.131 & 0.001 & 0.016 & 1.049 \\
9 & 0.003 & 0.177 & 0.002 & 0.001 & 0.148 & -0.011 & 0.021 & 1.038 \\
10 & 0.001 & 0.177 & 0.014 & 0.004 & 0.151 & -0.009 & 0.022 & 1.045 \\
11 & 0.0 & 0.177 & 0.02 & 0.005 & 0.165 & -0.005 & 0.026 & 1.048 \\
12 & 2.12 & 0.0 & 0.107 & 0.022 & -0.132 & -0.032 & 0.019 & 0.992 \\
13 & 2.137 & 0.007 & 0.003 & 0.002 & -0.025 & -0.009 & 0.001 & 0.982 \\
14 & 2.146 & 0.01 & 0.022 & 0.008 & -0.022 & -0.008 & 0.001 & 0.977 \\
  \end{tabular}

  \vspace{0.2cm}
  \begin{tabular}{c|c|c|c|c|c|c|c|c}
$n$&  $\dlmp_{n,1}$ &  $\dlmq_{n,1}$ &  $p_{n,1}^c$ &  $q_{n,1}$ &  $f_{n,1}$ & $ g_{n,1}$ & $\ell_{n,1}$ &  $v_{n,1}$  \\\hline
0 & 1.0 & 0.0 & -1.299 & -0.549 & - & - & - & 1 \\
1 & 1.002 & 0.004 & 1.05 & 0.245 & -0.924 & -0.321 & 1.057 & 0.906 \\
2 & 0.979 & 0.021 & 0.0 & 0.0 & 0.127 & -0.074 & 0.024 & 0.909 \\
3 & 0.942 & 0.046 & 0.037 & 0.015 & 0.131 & -0.072 & 0.024 & 0.916 \\
4 & 0.945 & 0.047 & 0.027 & 0.007 & -0.083 & -0.019 & 0.008 & 0.911 \\
5 & 0.947 & 0.048 & 0.031 & 0.008 & -0.057 & -0.013 & 0.004 & 0.909 \\
6 & 0.95 & 0.048 & 0.026 & 0.006 & -0.026 & -0.006 & 0.001 & 0.905 \\
8 & 0.004 & 0.176 & 0.006 & 0.001 & 0.254 & -0.035 & 0.07 & 0.932 \\
7 & -0.001 & 0.176 & -0.143 & 0.0 & 0.143 & 0.001 & 0.022 & 0.947 \\
9 & 0.003 & 0.176 & 0.036 & 0.022 & 0.118 & -0.033 & 0.016 & 0.933 \\
10 & 0.001 & 0.176 & 0.015 & 0.004 & 0.154 & -0.011 & 0.025 & 0.94 \\
11 & 0.0 & 0.176 & 0.026 & 0.006 & 0.169 & -0.006 & 0.03 & 0.943 \\
12 & 2.008 & 0.272 & 0.342 & 0.071 & -0.374 & -0.083 & 0.15 & 0.977 \\
13 & 2.031 & 0.281 & 0.002 & 0.001 & -0.032 & -0.012 & 0.001 & 0.965 \\
14 & 2.044 & 0.286 & 0.03 & 0.011 & -0.03 & -0.011 & 0.001 & 0.957 
  \end{tabular}
\end{scriptsize}
\caption{Results of the OPF for the 15-bus network for time periods t=0 (\textit{above}) and t=1 (\textit{below}).}
\label{tab:resNum}
    \end{table}
    \fi
     The optimal production for node $11$ is $\pRE_{11,0}= 0.185 $ and  $\pRE_{11,1}= 0.194 $, while the optimal costs obtained are $ c_0(\pRE_{0,0})=0.873$ and $c_1(\pRE_{0,1})=1.299$.
    Several comments are to be made.%

    First, one can observe that, even if the example does not satisfy the theoretical assumptions of ~\Cref{thm:farivarSOCP} or \Cref{thm:ganSOCPexactV}, the SOCP relaxation is exact and gives the optimal solution of the original OPF problem.

Second, the solutions show that the active (and reactive) DLMPs obtained for each time period are closed to the DLMPs at the root node $(\bm{\dlmp}_0, \bm{\dlmq}_0)$, except in two cases:
\begin{itemize}[wide]
\item for the branch composed of nodes $8,7,9,10,11$, the active DLMPS are close to 0.0 on all time periods. This is due to the fact that the renewable production of node 11 at cost 0 and the negative load of node $7$, which together fully compensate the demand on this branch. 
The  energy left is exported to the remaining nodes of the network up to the saturation of the line $(3,8)$. This saturation creates an important difference of the DLMPs between one part of the network and the other. The DLMPs here give incentives for the nodes to either decrease production or increase consumption as much as possible, in order to release the saturation of line $(3,8)$.
\item On the contrary, the active DLMPs on the branch composed of nodes $(12,13,14)$ on time period $t=1$ are twice as large as the other nodes. Again, this is explained by the saturation of line $(0,12)$.
 Because of the cheaper price on time period $1$, LAs are encouraged to move the flexible demand onto this period, which results in the congestion of line $(0,12)$. Here, the DLMPs counterbalance this price difference  and give incentives to decrease consumption on $t=1$ for nodes $12,13,14$ in order to stay within the line capacity.
\end{itemize}

Third, the DLMP for node 7 and $t=1$ is strictly negative: at this time period, the (negative) consumption for this node is at its upper bound $p_{7,1}=\overline{P}_{7,1}=-0.143$.  The negative DLMP shows that the system will be better off if less power was injected by node 7. Interestingly, this case shows that, in some specific cases, DLMPs can be negative, even if we consider an increasing cost function $\phi_0$.

\begin{table}[H]
  \centering
  \begin{tabular}{|c|c|c|c|}
    Agg. & Nodes & DLMP payment & VCG payment \\\hline\hline
    1  &  [1, 2, 3]  &  2.464  &  2.04 \\\hline
2  &  [4, 5, 6, 12, 13]  &  0.693  &  0.675 \\\hline
3  &  [8, 7, 14]  &  0.077  &  -0.007 \\\hline
4  &  [9, 10]  &  0.006  &  0.004 \\\hline
5  &  [11]  &  0.002  &  -0.115 \\\hline
  \end{tabular}
  \caption{Comparison of DLMP payments and VCG payments}
  \label{tab:dlmpVSvcg}
\end{table}
\Cref{tab:dlmpVSvcg} shows the difference between the DLMP payments of each aggregator $a$, and its VCG payments, given in \Cref{sec:vcg-meca}. It is interesting to observe that DLMP payments remain close to VCG payments. We also observe on this example that the DLMP payment of each aggregator is always larger than the VCG payment.
This inequality was proved to hold systematically in \cite{samadi2012advanced} in a much simpler framework (without network constraints and power upper bounds). Extending it in the framework of this paper would be an interesting avenue.

\begin{figure}[H]
\begin{subfigure}{0.49\columnwidth} \centering
 \includegraphics[width=\textwidth]{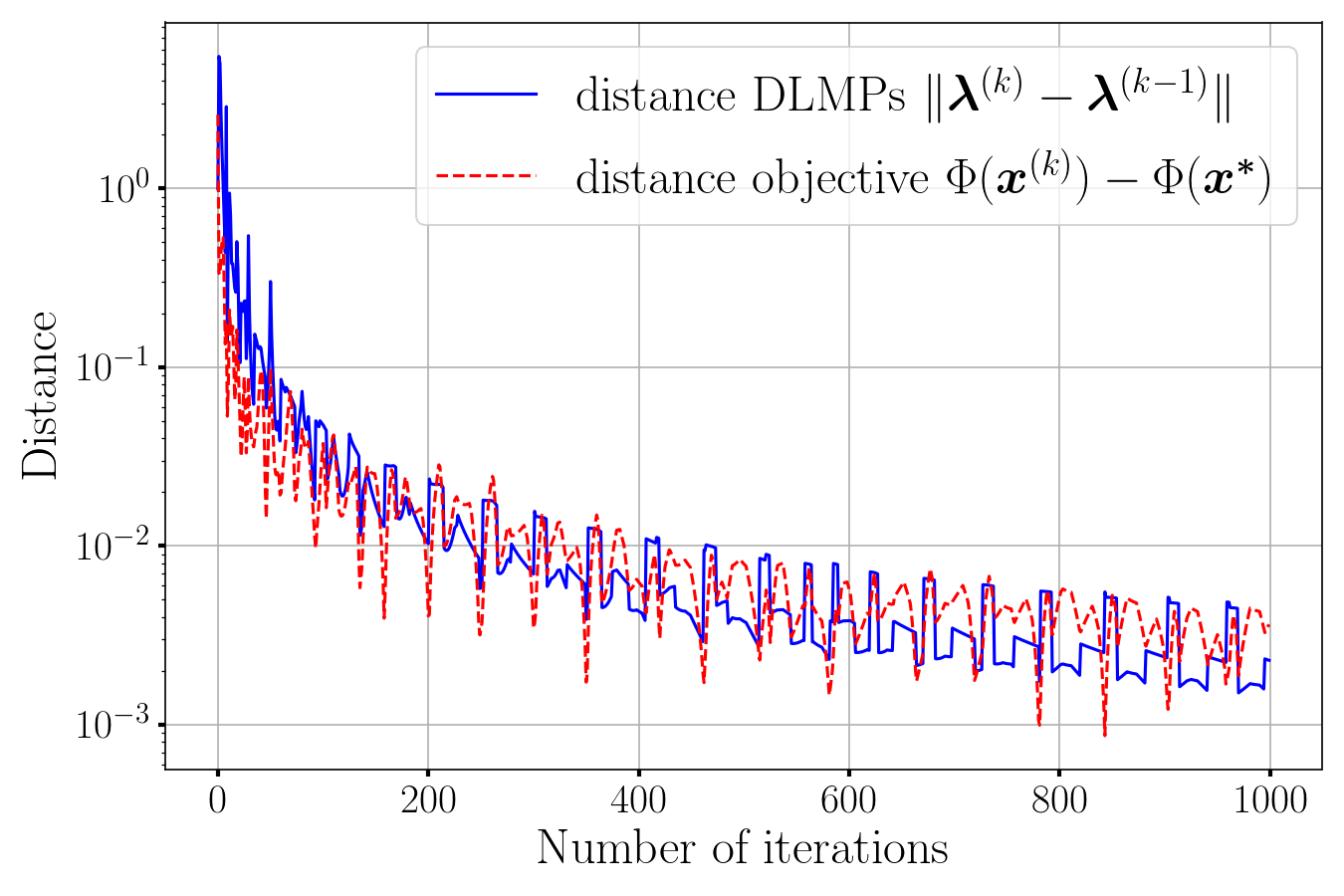}
 \subcaption{ PDGS ($K=4$)}
\end{subfigure}
\begin{subfigure}{0.49\columnwidth} \centering
\includegraphics[width=\textwidth]{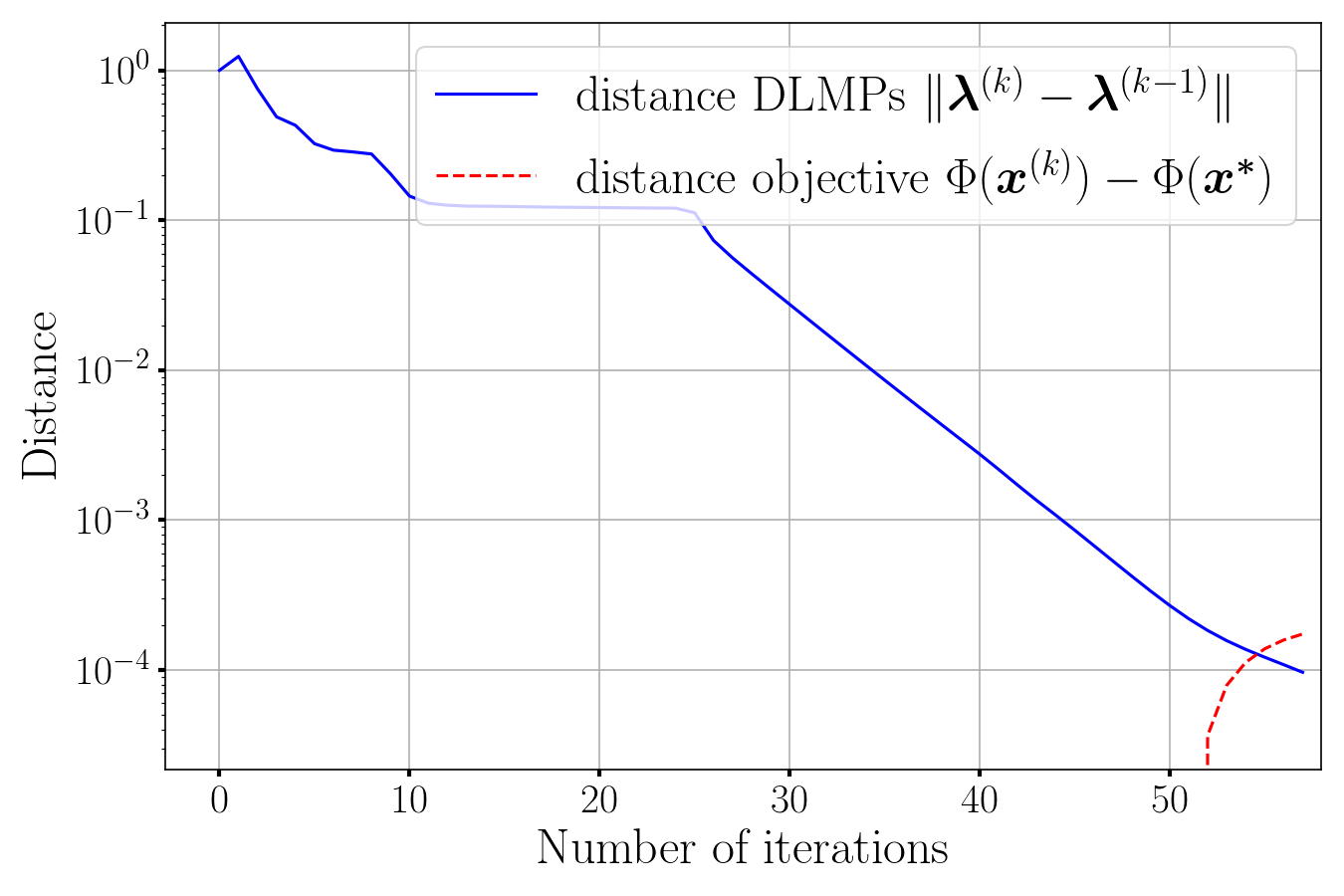}
\subcaption{ ADMM ($\rho=5$)}
\end{subfigure}
 \caption{Convergence of PDGS (\Cref{algo:dualDecompAltMin}) and ADMM (\Cref{algo:dualAscent}). \emph{In spite of other advantages, PDGS is very slow to converge after a few iterations.}}
 \label{fig:algosCvg}
\end{figure}

In \Cref{fig:algosCvg}, we compare the convergence of  PDGS (\Cref{algo:dualDecompAltMin}) and ADMM (\Cref{algo:dualAscent}).
For PDGS in this example, infeasibility of the DSO subproblem (see \Cref{algo:dualDecompAltMin}) only appears on iterations 1,2,12 and, ultimately, 19: the DSO subproblem is always feasible afterwards.
PDGS is very slow to converge: as opposed to ADMM which needed 60 iterations (executed in 21.41 seconds in our configuration) to converge with an error of less than $10^{-4}$ on primal feasibility (which is equal to the primal feasibility error $\rho( A_0\xx_0^{(k+1)} + B\xx_{\A}^{(k+1)}\hh -\bm{b})$)  and objective, PDGS is still more than $10^{-3}$ away from the optimal objective cost after 1000 iterations (executed in 364.23 seconds in our configuration).
Although a parallel implementation could save execution time (it is possible to solve all LAs' problems in parallel, and to solve all time periods problems in parallel in the DSO's problem), it could be necessary to rely on ADMM rather than on PDGS in practice, for a larger network and a larger number of time periods.

A possibility for the DSO could be to rely on the two algorithms:  considering the last solutions of the LAs for some iterations of ADMM, and then applying PDGS until reaching primal feasibility.

Improving the convergence of PDGS, for instance by adding a regularizing term in the subproblems, could also be an avenue for further research.

    \section*{Conclusion}

    We proposed a coordination procedure for aggregators operating on a distribution network, which respects the decentralized structure of decisions and information.
    This procedure enables to compute decentralized decisions  satisfying AC network constraints and  leading to  an optimal grid operation from the system operator point of view.

    Several directions of research are interesting for extending this work.
    One could study the tractability and the limits of the proposed procedure for large-scale instances, considering  networks with a large number of nodes and/or a large set of time periods.   

    Second, an interesting but complex extension  to the proposed model would be to consider a strategic and game-theoretic framework in which  aggregators are subject to the DSO incentives, but are also price-makers on an electricity market.

    \ifreferonline
    \else
        \section*{Acknowledgments}
    The author sincerely thanks St\'ephane Gaubert, Nadia Oudjane, Olivier Beaude and Alejandro Jofr\'e for precious discussions and insightful comments.

    \fi

\renewcommand{\bibfont}{\small}
\bibliographystyle{myIEEEtranN}
\bibliography{../../../USEFULPAPERS/Biblio_complete/shortJournalNames,../../../USEFULPAPERS/Biblio_complete/biblio1,../../../USEFULPAPERS/Biblio_complete/biblio2,../../../USEFULPAPERS/Biblio_complete/biblio3,../../../USEFULPAPERS/Biblio_complete/biblio4,../../../USEFULPAPERS/Biblio_complete/biblioBooks}

\newcommand{\mum}{\underline{\mu}}
\newcommand{\mup}{\overline{\mu}}

\end{document}